\pdfoutput=1
\documentclass[a4paper,12pt]{article}
\usepackage[english]{babel}
\usepackage{amsfonts}
\usepackage[T1]{fontenc}
\usepackage[utf8]{inputenc}  
\usepackage{graphicx}  
\usepackage{amsmath}
\usepackage{array}
\usepackage{enumerate}                                                                                 
\usepackage{amscd}
\usepackage{stmaryrd}
\usepackage{amsthm}
\usepackage{mathabx}
\usepackage{pstricks}
\usepackage{dsfont}
\usepackage{lmodern}
\usepackage{epsfig}
\usepackage{pstricks-add}
\usepackage{fancyhdr}
\usepackage{geometry} 
\usepackage{pgf,tikz}
\usepackage{xcolor}
\usepackage{pgf,tikz}
\usepackage{mathrsfs}
\usepackage{amsfonts}
\usepackage{amssymb}
\usepackage{amstext}
\usepackage{dsfont}
\usepackage[font={small,it}]{caption}
\usepackage{authblk}
\usepackage{listings}
\usetikzlibrary{arrows}
\usetikzlibrary{arrows}

\DeclareMathOperator{\R}{\mathbb{R}}

\DeclareMathOperator{\N}{\mathbb{N}}

\theoremstyle{plain}
\newtheorem{thm}{Theorem}[section] 

\newtheorem{biol}[thm]{Biological interpretation}
\newtheorem{nota}[thm]{Notation}
\newtheorem{prop}[thm]{Proposition}
\newtheorem{theo}[thm]{Theorem}
\newtheorem{defi}[thm]{Definition}
\newtheorem{rema}[thm]{Remark}
\newtheorem{lemm}[thm]{Lemma}

\title{A birth-death model of ageing: from individual-based dynamics to evolutive differential inclusions}

\author[1]{Sylvie Méléard}
\author[2]{Michael Rera}
\author[3]{Tristan Roget}

\affil[1]{\'Ecole Polytechnique}
\affil[2]{Sorbonne Université}
\affil[3]{Université de Montpellier}

\begin{document}
\maketitle
\begin{abstract}
Ageing’s sensitivity to natural selection has long been discussed because of its apparent negative effect on an individual's fitness. Thanks to the recently described (Smurf) 2-phase model of ageing (\cite{tricoire2015new}) we propose a fresh angle for modeling the evolution of ageing. Indeed, by coupling a dramatic loss of fertility with a high-risk of impending death - amongst other multiple so-called hallmarks of ageing - the Smurf phenotype allowed us to consider ageing as a couple of sharp transitions.
The birth-death model (later called bd-model) we describe here is a simple life-history trait model where each asexual and haploid individual is described by its fertility period $x_b$ and survival period $x_d$. We show that, thanks to the Lansing effect, the effect through which the “progeny of old parents do not live as long as those of young parents”, $x_b$  and $x_d$ converge during evolution to configurations $x_b-x_d\approx 0$ in finite time. 
\\To do so, we built an individual-based stochastic model which describes the age and trait distribution dynamics of such a finite population. Then we rigorously derive the adaptive dynamics models, which describe the trait dynamics at the evolutionary time-scale. We extend the Trait Substitution Sequence with age structure to take into account the Lansing effect. Finally, we study the limiting behaviour of this jump process when mutations are small. We show that the limiting behaviour is described by a differential inclusion whose solutions $x(t)=(x_b(t),x_d(t))$ reach the diagonal $\lbrace x_b=x_d\rbrace$ in finite time and then remain on it. This differential inclusion is a natural way to extend the canonical equation of adaptive dynamics in order to take into account the lack of regularity of the invasion fitness function on the diagonal $\lbrace x_b=x_d\rbrace$. 
\end{abstract}

\section{Introduction}
Ageing is commonly defined as an age-dependant increase of the probability to die after the maturation phase (\cite{kirkwood2000we}). It affects a broad range of organisms in various ways ranging from negligible senescence to fast post-reproductive death (reviewed in \cite{jones2014diversity}). In the recent years, a new 2-phases model of ageing proposed by \cite{tricoire2015new} described the ageing process not as being continuous but as made of at least 2 consecutive phases separated by a dramatic transition. This transition, dubbed “Smurf transition”, was first described in drosophila (\cite{rera2011modulation}, \cite{rera2012intestinal}). In short, this transition occurs in every individuals prior to death and is marked by a series of associated phenotypes encompassing high-risk of impending death, increased intestinal permeability, loss of energy stores, reduced fertility (\cite{rera2012intestinal}). It was later showed to be evolutionarily conserved in \textit{Caenorhabditis elegans} and \textit{Danio rerio} (\cite{dambroise2016two}, \cite{rera2018smurf}). Such broad evolutionary conservation of a marker for physiological age raises the question of an active selection of the underlying mechanisms throughout evolution.
Since the beginning of ageing studies, the question of its ability to appear through evolution has been raised. In fact, since the Darwinian theory of evolution stipulates that species arise and develop thanks to the natural selection of small, inherited variations that increase an individual's ability to compete, survive, and reproduce (\cite{bhl124544}), numerous researchers suggested that ageing - and more precisely senescence - could not be actively and directly selected thanks to evolution (\cite{fabian2011evolution}). One of the first to publicly address the question of the evolution of ageing was August Weismann who proposed in 1881 that the life expectancy was programmed by “the needs of the species” (\cite{weismann1881origin}). Numerous theoretical works have been developed about ageing for the past 60 years in order to recenter the selection of an ageing process on the individuals more than on the population. 
Here we will focus our attention on the capability of a process such as ageing to be selected through evolution.
\\If fitness alone - as an individual's reproductive success or its average contribution to the gene pool of the next generation - were at play in the evolution process, the best adapted individuals would have infinite fertility as well as longevity. Nevertheless, this situation is never observed mainly because organisms adapted to constant variations of environmental conditions and physical limitations of resources availability. Thus, an active mechanism for the elimination of these fitness-excessive individuals would represent a selective advantage in an environment where scarcity is the rule. The Lansing effect is a good candidate for such a mechanism. It is the effect through which the “progeny of old parents do not live as long as those of young parents” first described in rotifers (\cite{lansing1947transmissible}, \cite{lansing1954nongenic}). More recently, it has been shown that older drosophila females and in some extent males tend to produce shorter lived offspring (\cite{priest2002role}), zebra finch males give birth to offspring with shorter telomere lengths and reduced lifespans (\cite{noguera2018experimental}) and finally in humans, “Older father’s children have lower evolutionary fitness across four centuries and in four populations” (\cite{arslan2017older}).
\\In the present article, we decided to approach the problem of ageing selection and evolution by using an extremely simplified version of a living organism. It is an haploid and asexual organism carrying only two traits, $x_b$ that defines the duration of its ability to reproduce and $x_d$ that defines the duration of its ability to maintain its integrity - stay alive (see Figure \ref{fig:histoiredevie}). We will further discuss the properties of this simple model in the next part. Although quite simple, it allows the modeling of all types of observed ageing modes : negligible senescence, sudden post-reproductive death, or post-reproductive “menopause-like” survival as well as the smurf phase.
\\The main result of the present article is that a pro-senescence program can be selected through Darwinian mechanisms thanks to the Lansing effect. Indeed, our main mathematical result (see Theorem \ref{theo:diffinc}) shows that evolution drives the trait $(x_b,x_d)$ towards configurations $x_b = x_d$. It means that the individuals can enjoy all their reproductive capacity, and then are quickly removed from the population. 
Moreover, this result shows that after reaching the configurations $x_b=x_d$, the traits $x_b$ and $x_d$ continue to increase with decreasing speed,  while maintaining $x_b=x_d$. This decrease in the speed of evolution is a consequence of the fitness gradients being decreasing functions of the traits (see Remark \ref{remacaswell2}). It is related to the well-known fact that the strength of selection decreases with age, i.e that  a mutation having an effect on the reproduction or mortality rates at a given age will have all the more impact as this age is small (\cite{haldane1942new}, \cite{medawar1952unsolved}, \cite{hamilton1966moulding}). Indeed, in our model a perturbation of the trait $x_b$ (resp. $x_d$) is equivalent to a perturbation of the birth rates at age $x_b$ (resp. death rates at age $x_d$). 
\\We build an individual based stochastic model inspired by \cite{tran2008large}. It describes an asexual and haploid population with a continuous age and a continuous life-history trait structure. In this model, the life-history trait of every individual is thus a pair of positive numbers $(x_b, x_d)\in\R_+^2$. An individual with trait $(x_b,x_d)$ reproduces at rate one as long as it is younger than $x_b$ and cannot die as long as it is younger than $x_d$ (see Figure \ref{fig:histoiredevie}).
\begin{figure}[!h]
\begin{center}
\includegraphics[scale=0.2]{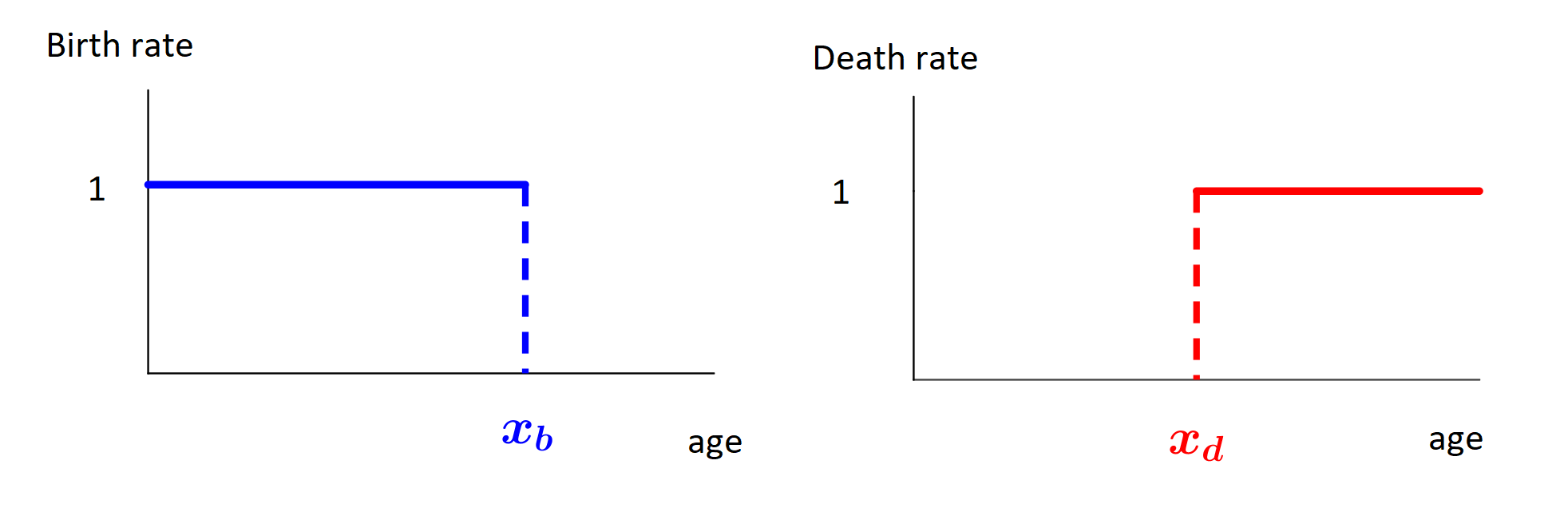}
\caption{Life-history associated to the trait $x=(x_b,x_d)$.}
\label{fig:histoiredevie}
\end{center}
\end{figure}
This model leads to three configurations: $x_d<x_b$, $x_d=x_b$ and $x_d>x_b$ (see Figure \ref{fig:fig:classes}).
\begin{figure}[h!]
\begin{center}    
\includegraphics[scale=0.13]{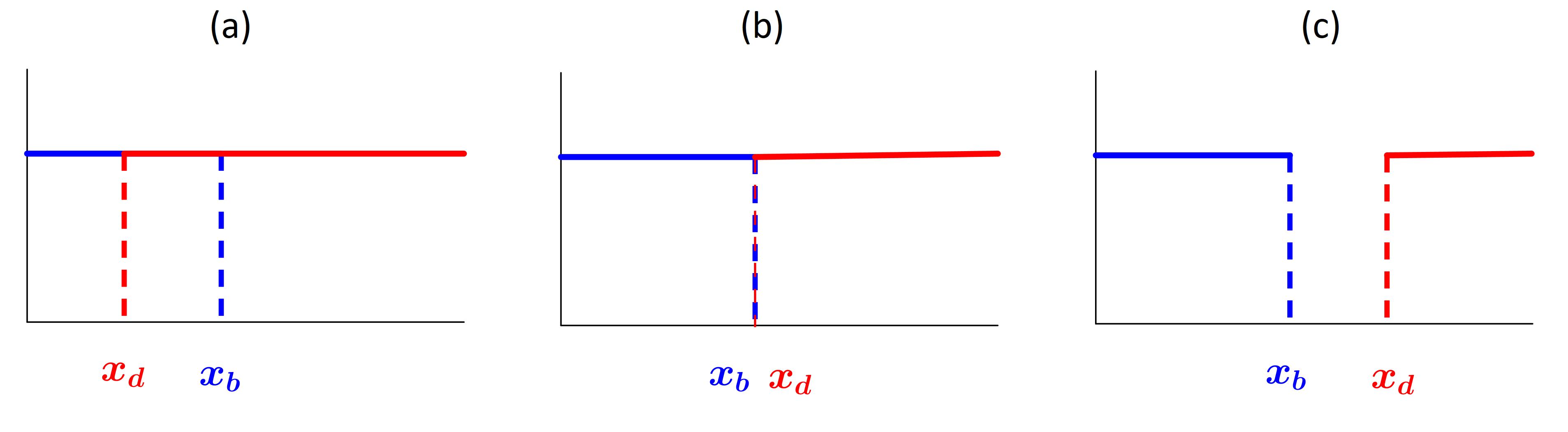}
        \caption{Three typical configurations of the model.  $(a)$ 'Too young to die' : it corresponds to configurations $(x_b,x_d)$ which satisfies $x_d<x_b$ ; $(b)$ 'Now useless': it corresponds to configurations $(x_b,x_d)$ which satisfy $x_b=x_d$; $(c)$ 'Menopause-like': it corresponds to configurations $(x_b,x_d)$ which satisfies $x_d>x_b$.}
\label{fig:fig:classes}    
    \end{center}
\end{figure}
From one generation to the next, variation on the trait $(x_b,x_d)$ is generated through genetic mutations. In addition, natural selection occurs through mortality due to competition for resources thanks to a logistic equation defining the maximum load of the medium. Finally, we model an epigenetic effect of senescence through the Lansing Effect. It introduces a source of phenotypic variation at a much faster time-scale than genetic mutations. In that aim, we assume that an individual that reproduces after age $x_d$ transmits to his descendant a shorter life expectancy (see Section 2 for details). Therefore, only individuals with trait $x_d < x_b$ are affected (see Figure \ref{fig:fig:classes} (c)). That creates an adaptive trade-off which impacts the phenotypic evolution of the population. 
\\The purpose of the present article is to study the long-term evolution of the trait $x=(x_b,x_d)$ and to determine whether it concentrates on $x_b -x_d = 0$. To do so, we are inspired by the theory of adaptive dynamics (\cite{metz1992should}, \cite{metz1996adaptive}, \cite{dieckmann1996dynamical}) which studies the phenotypic eco-evolution of large populations under the assumption that genetic mutations are rare and have small effects. A central tool in that theory is the concept of \textit{invasion fitness}. The invasion fitness is a function $1-z(y,x)$ informally defined as the probability that an individual with trait $y$ survives in a resident population with trait $x$ at demographic  equilibrium. In section 4, we prove that the invasion fitness satisfies the simple relation $1-z(y,x)=(\lambda(y) -\lambda(x))\vee 0$ where $\lambda(x)$ is the \textit{Malthusian parameter}, describing the adaptive value associated with the trait $x$ (see Section 3.1 (\ref{eq:malthus}) for the definition). This allows us to introduce the \textit{Trait Substitution Sequence} process (TSS) which is a pure jump process describing the successive invasions of successfull mutants in monomorphic populations at the demographic equilibrium. The TSS has been heuristically introduced in \cite{metz1996adaptive}, \cite{dieckmann1996dynamical} for a trait structured population. In \cite{champagnat2006microscopic}, it has been rigorously derived from an individual based model and generalised in \cite{meleard2009trait}, to age-structured populations. Our case differs from \cite{meleard2009trait} by mainly two aspects: the additional Lansing effect and the specific form of the mutations kernel which is not absolutely continuous with respect to the Lebesgue measure on $\R^2$ (as assumed in \cite{champagnat2006microscopic},  \cite{meleard2009trait}). 
In the usual case, the TSS is approximated by the solution of the \textit{Canonical equation of adaptive dynamics} when the size of mutations is small and on a longer time-scale (\cite{champagnat2002canonical}, \cite{champagnat2011polymorphic}, \cite{meleard2009trait}). This limit theorem requires the Lipschitz regularity of the fitness gradient. In our model this assumption is not satisfied. Nonetheless, we prove that the limiting behaviour of the TSS when mutation are small is captured by a differential inclusion, using an approach  developed in \cite{gast2012markov}. A differential inclusion is an extension of ordinary differential equation to set-valued time-derivatives, which extends Cauchy-Lipschitz theory to non regular gradient cases. In our model, the gradient is smooth except on the diagonal $\lbrace x_b=x_d\rbrace$. We prove that the solutions are well-defined until they reach the diagonal which they do in finite time. Indeed, the drift towards the diagonal is due to on one side $(x_b<x_d)$ to the fact that the individuals with larger $x_b$ will reproduce more and thus tend to invade ; and on the other side $(x_d<x_b)$ to the Lansing effect because old individuals produce short-lived offsprings. Hence, there is no advantage associated with an increasing $x_b$, only when associated with an increasing $x_d$ that maximizes survival. Then the trait $(x_b,x_d)$ evolves following the diagonal. The drift of the differential inclusion depends on the derivatives of the Malthusian parameter with respect to the trait variable. These derivatives are expressed as functions of the stable age distribution and reproductive value as in \cite{hamilton1966moulding}, \cite{caswell2010reproductive} (see Remark \ref{remacaswell2}).
\vspace{0.2cm}
\\In Section 2, we present the individual-based  model. Thanks to simulations, we show what was suggested by observations: the trait distribution of the population stabilises on the diagonal $x_b-x_d=0$.
\\ In Section 3, we study the deterministic approximations of the stochastic dynamics under the assumption of large population and rare mutations. These approximations are non-linear systems of partial differential equations similar to the Gurtin-McCamy Equation (\cite{gurtin1974non}). We study their long-time behaviour and give some results of convergence to the stationary states.
\\In Section 4,  we state and prove the main mathematical results of this article concerning the approximation by the TSS and the canonical inclusion of adaptive dynamics. (Theorem \ref{th:approxtss} and Theorem \ref{theo:diffinc}).
\\ In Section 5, we give some comments on our results.
\section{A stochastic model for the evolution of life-history traits}
At time $t\geq 0$, the population is described by a point measure on $(\R_+^{*})^{2}\times \R_+$
\begin{equation}\label{eq:measurepop}
Z_t^K(dx,da)=\frac{1}{K}\sum_{i=1}^{N_t^K}\delta_{(x^{i}(t),a^{i}(t))}(dx,da)
\end{equation}
where $N_t^K = K\langle Z_t^K,1\rangle$ is the total population size, $K$ is the order of magnitude of a population at equilibrium, $x^{i}(t)=(x_b^i(t),x_d^i(t))\in (\R_+^{*})^2$ is the trait of the individual $i$ and $a^i(t)\in\R_+$ is its age. The dynamics is defined as a piecewise deterministic Markov process which jumps as follows:
\begin{itemize}
\item[•]An individual $(x,a)\in (\mathbb{R}_+^{*})^2\times\R_+$ reproduces at rate $\mathds{1}_{a\leq x_b}$. The trait of the descendant $y=(y_b,y_d)$ is determined by the following two-steps mechanism (see Figure \ref{fig:noyau} below for an illustration):
\begin{itemize}
\item[-]Step 1: If $a\leq x_d$, the offspring inherits the trait $x=(x_{b},x_{d})$.\\
\emph{Lansing Effect}: if $a>x_d$, we assume that the offspring carries the trait  $(x_b,0)$.
Let us denote by $\tilde{x}$  the trait  defined as $\tilde{x}=x$ if $a\leq x_d$ and  $\tilde{x}=(x_b,0)$ if $a>x_d$. We observe that $x_{b}$ remains unchanged and that only individuals with configurations $x_d<x_b$ are concerned by the second type $(x_b,0)$ (see Figure \ref{fig:fig:classes} (c)).
\item[-]Step 2: Genetic mutations. A mutation appears instantaneously on each trait ${x}_b$ and $\tilde{x}_d$ independently with probability $p_K\in \left]0,1\right[$. If the trait ${x}_b$ mutates, the trait $y_b$ of the descendant is $y_b=x_b +h_b$ where $h_b\in\R$ is chosen according to the probability measure $k(x_b,h_b)dh_b$; if the trait $\tilde{x}_d$ mutates, the trait $y_d$ of the newborn is $y_d=\tilde{x}_d+h_d$ where $h_d\in\R$ is chosen according to $k(\tilde{x_d},h_d)dh_d$, where the mutational kernel $k$ is defined for all $u\in\mathbb{R}_+$ and $v\in \R$ by
\begin{equation}
k(u,v)=\frac{\mathds{1}_{\left[0\vee(u-1),u+1\right]}(u + v)e^{-\frac{v^2}{\sigma^2}}}{\int_{\R}\mathds{1}_{\left[0\vee(u-1),u+1\right]}(u + z)e^{-\frac{z^2}{\sigma^2}}dz}
\end{equation}
where $\sigma^2>0$. Note that for $u>1$, 
\begin{equation}\label{eq:k}
k(u,v)=k(v)= \frac{\mathds{1}_{\left[-1,1\right]}( v)e^{-\frac{v^2}{\sigma^2}}}{\int_{\R}\mathds{1}_{\left[-1,1\right]}(z)e^{-\frac{z^2}{\sigma^2}}dz}.
\end{equation}
\end{itemize}
\item[•]An individual with trait $(x,a)$ has a death rate $\mathds{1}_{a>x_d}+\eta N_t^K$, with $\eta>0$, meaning that  each individual is subjected to the same competition pressure $\eta$ from any individual in the population and whatever the value of its trait.
\item[•]Between jumps, individuals age at speed one: an individual with age $a$ at time $t$ has an age $a+s$ at time $t+s$.
\end{itemize}
Figure \ref{fig:noyau} summarizes the trait dynamics described above. 
\begin{figure}[!h]
\begin{center}
\includegraphics[scale=0.13]{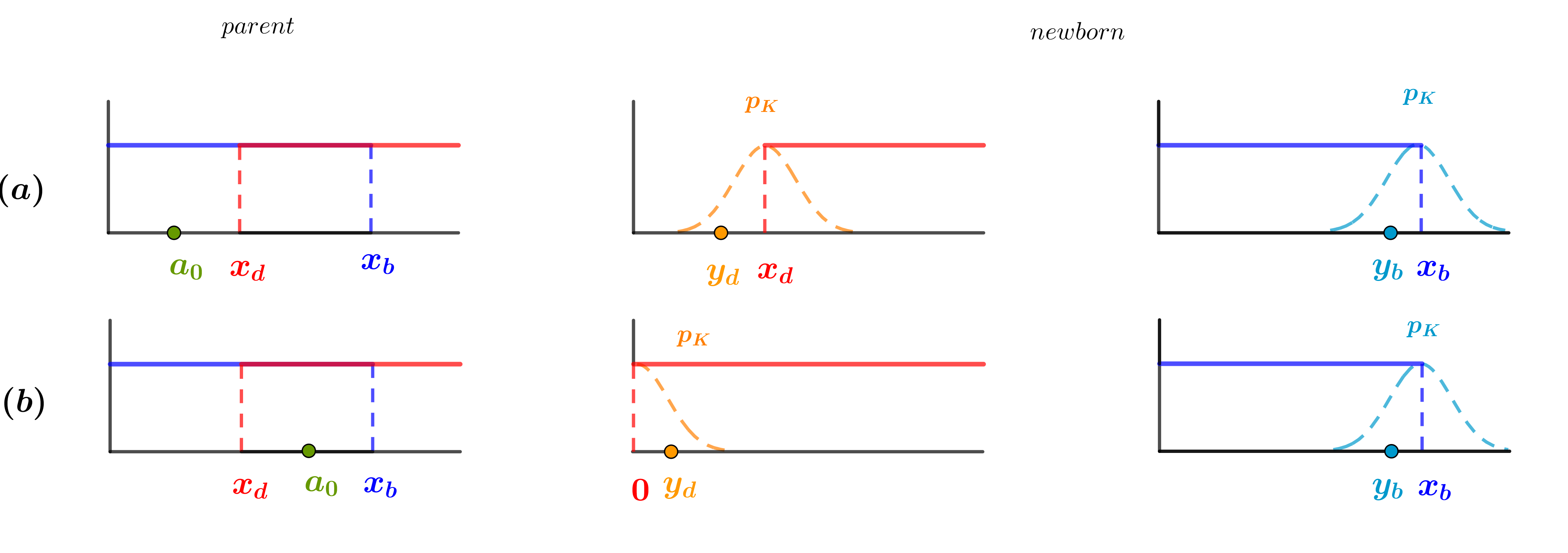}
\caption{Picture of the reproduction and mutation mechanism. (a): the individual reproduces at age $a_0<x_d$, there is no Lansing Effect. (b): the individual reproduces at age $a_0>x_d$, the Lansing Effect acts. }
\label{fig:noyau}
\end{center}
\end{figure}
\paragraph{Numerical simulation.} 
 The pictures in Figure \ref{fig:simumicro}  represent a simulation of the trait marginals dynamics  of the process $Z_t^K(dx,da)$. We consider a monomorphic initial population with trait $x=(1.2, 2.5)$ and $N_0^K=10000$. We consider a competition rate $\eta=0.0005$, a probability of mutation $p_K=0.05$ and a variance of mutations $\sigma=0.05$.
\begin{figure}[h!]
\begin{center}
\includegraphics[scale=0.3]{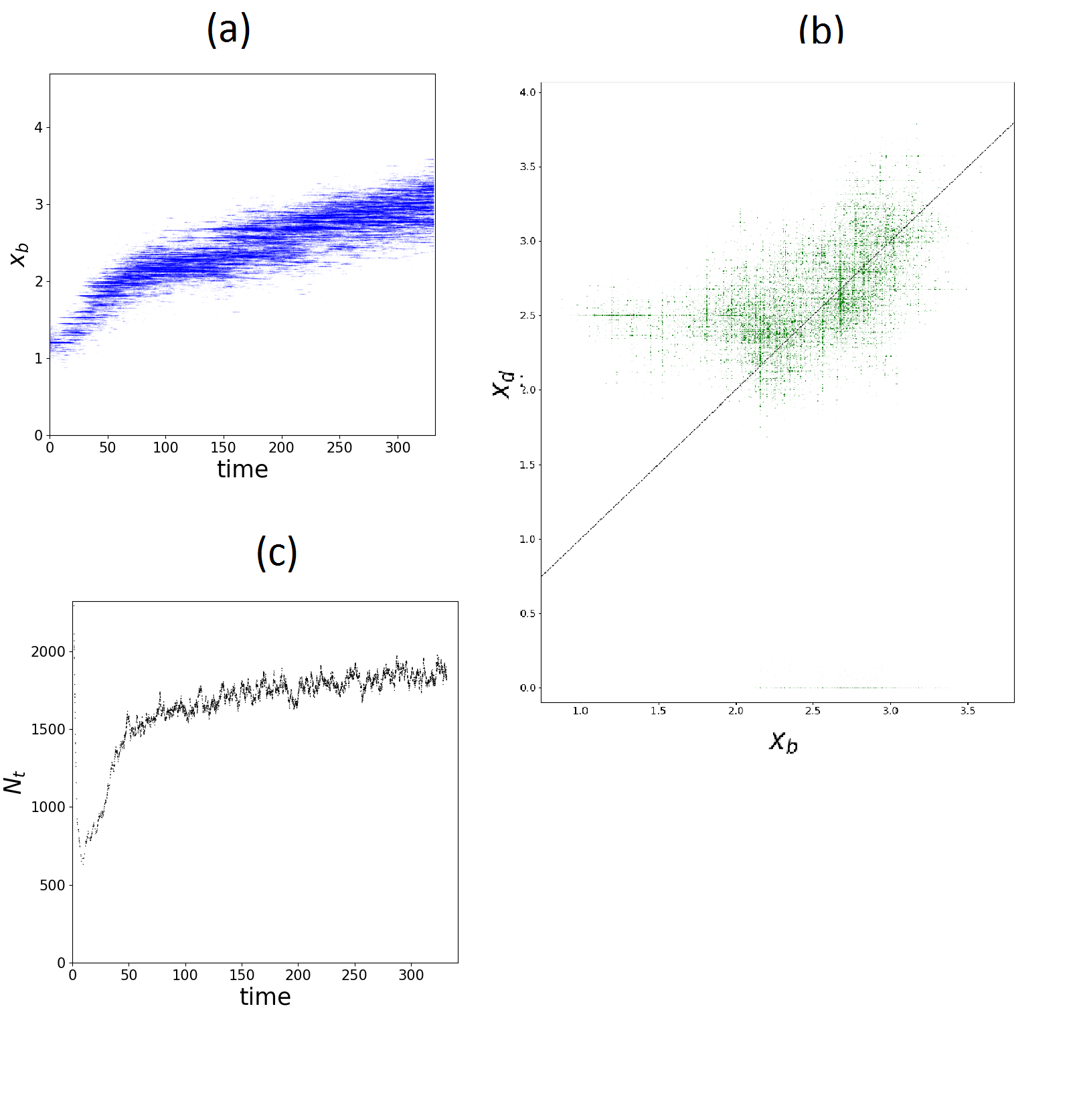}
\end{center}
\caption{Simulation of the individual based model (see the script in Appendix A.4) . (a): Dynamics of the trait $x_b$ as a function of time. (b): Dynamics of the trait $x_d$ as a function of the trait $x_b$. (c): Population size as a function of time. Parameters: $N_0^K=10000$ individuals with trait $(1.2,2.5)$, $\eta=0.0005$, $p=0.05$, $\sigma=0.1$.}
\label{fig:simumicro}
\end{figure}
\newline
At time $t=0$, the population is monomorphic with trait $(x_b,x_d)=(1.2,2.5)$. We observe that before the trait $x_b$ reaches the value of $2.5$, the trait $x_d$ remains constant. When $x_b$ reaches $x_d\simeq 2.5$, the traits $x_b$ and $x_d$ continue to increase by maintaining $x_b\simeq x_d$.
\section{Monomorphic and bimorphic deterministic dynamics}
In this section, we study a deterministic approximation of the process $Z^K$ when $K$ goes to infinity. We also assume that $p_K$ goes down to zero: almost no mutation occurs on a time interval $\left[0,T\right]$. Nonetheless, some phenotypic variation is created by Lansing Effect. Since our model is density-dependent, the deterministic approximation is a system of classical non-linear partial differential equations similar to the Gurtin MacCamy Equation  (\cite{gurtin1974non}). In the monomorphic case, we show that the dynamics converges to the unique non-trivial equilibrium. In the bimorphic case, we show the convergence to a monomorphic equilibrium. We will consider that a monomorphic population with trait $x$ is composed of two subpopulations with traits $(x_b,x_d)$ and $(x_b,0)$.
\begin{nota}
We define $\textbf{I}=\left(\begin{array}{cc}
1 & 0 \\ 
0 & 1
\end{array}\right)$.
\end{nota}
\subsection{Monomorphic dynamics}
Let $x=(x_b,x_d)\in (\R_+^{*})^2$ be a phenotypic trait. We consider a monomorphic initial population such that the sequence $(Z_0^K)_{K}$ weakly  converges as $K\rightarrow \infty$ to $\delta_x n_x(a)da$ which describes a monomorphic population with trait $x$. Then, as in \cite{tran2008large}, we can prove that the sequence of processes $(Z^K)_{K}$ converges in probability on any finite time interval to the weak solution $(n_x(t,.),t\geq 0)=((n_x^1(t,.),n_x^2(t,.)),t\geq 0)\in C(\R_+, L^1(\R_+)^2)$ of the following system of partial differential equations
\begin{equation}\label{eq:edpmono}
\begin{cases}
\partial_t n_x(t,a) +\partial_a n_x(t,a) = -\left( \textbf{D}_x(a) +\eta \Vert n_x(t,.)\Vert_1 \textbf{I}\right)n_x(t,a)\\
n_x(t,0)=\int_{\R_+}\textbf{B}_x(\alpha)n_x(t,\alpha)d\alpha
\end{cases}
\end{equation}
with initial condition given by $n_x(0,a)=n_x(a)$, where the densities $n^1_x(t,.)$ and $n^2_x(t,.)$ describe the population distributions with trait $(x_b,x_d)$ and $(x_b,0)$ respectively; 
\begin{equation*}
\Vert n_x(t,.)\Vert_1 = \sum_{i\in\lbrace 1,2\rbrace}\int_{\R_+}\vert n_x^{i}(t,\alpha)\vert d\alpha
\end{equation*}
is the total population size and
\begin{equation}\label{eq:birthdeathrate}
\textbf{B}_x(a)=\left(\begin{array}{cc}
\mathds{1}_{a\leq x_b\wedge x_d} & 0 \\ 
\mathds{1}_{x_d<a\leq x_b} & \mathds{1}_{a\leq x_b}
\end{array}\right),\quad 
\textbf{D}_x(a)=\left(\begin{array}{cc}
\mathds{1}_{a>x_d} & 0 \\ 
0 & 1
\end{array}\right)
\end{equation}
are the birth and death interactions. We refer to \cite{webb1985theory} for the well-posedness theory of $L^1(\R_+)^2$ solutions of Equation (\ref{eq:edpmono}). 
\begin{rema}
Let us comment the different terms in Equation (\ref{eq:edpmono}). The transport term on the left-hand side of the first equation describes the aging of the individuals, the right-hand side describes the death of the individuals. The renewal condition in $a=0$ describes the births (i.e the individuals with age $0$).
\end{rema}

\medskip \noindent We introduce the set of viable traits
\begin{equation}\label{eq:viableset}
\mathcal{V}=\lbrace x=(x_b,x_d)\in(\R_+^{*})^2: x_b\wedge x_d > 1\rbrace.
\end{equation} 
We show that for any trait $x\in\mathcal{V}$, there exists a unique non-trivial and globally stable stationary state of (\ref{eq:edpmono}). Indeed, the quantity $x_b\wedge x_d$ represents the mean number of descendants with trait $(x_b,x_d)$ of an individual with trait $(x_b,x_d)$ if there is no competition.
\begin{prop}\label{prop:monoedp}
Assume that $x\in\mathcal{V}$. There exists a unique non-trivial stationary solution $\overline{n}_x\in L^1(\R_+)^2$ to Equation (\ref{eq:edpmono}). Moreover, any $L^1(\R_+)^2$ non negative solution $n_x(t,.)$ of (\ref{eq:edpmono}) such that $n^1_x(t,.)\neq 0$ converges to $\overline{n}_x$ in $L^1(\R_+)^2$ as $t\rightarrow \infty$. 
\end{prop}
\begin{rema}
In Proposition \ref{pr:explicitstatio} below, we will give an explicit expression for the stationary state $\overline{n}_x$.
\end{rema}
The proof of Proposition \ref{prop:monoedp} is based on the  study of the associated linear dynamics. We introduce the linear operator
\begin{equation}\label{eq:linearop}
\begin{array}{ccc}
A:D(A)\subset L^1(\R_+)^2 & \rightarrow & L^1(\R_+)^2 \\ 
u & \mapsto & -u' - \textbf{D}_x(a)u
\end{array} 
\end{equation}
where $D(A)=\lbrace u\in L^1(\R_+)^2: u \text{ abs. cont., }u'\in L^1(\R_+)^2, \text{ }u(0)=\int_{\R_+}\textbf{B}_x(\alpha)u(\alpha)d\alpha\rbrace$. It is well known that $A$ is the infinitesimal generator of a strongly continuous semigroup of linear operators \cite[Proposition 3.7]{webb1985theory} which describes the solutions of the linear system of McKendrick Von-Foerster Equation
\begin{equation}\label{eq:monolinear}
\begin{cases}
\partial_t v_x(t,a) +\partial_a v_x(t,a) = -\textbf{D}_x(a)v_x(t,a)\\
v_x(t,0)=\int_{\R_+}\textbf{B}_x(\alpha)v_x(t,\alpha)d\alpha.
\end{cases}
\end{equation}
In \cite{clement2017analysis}, a similar linear model is studied. The "entropy method" (\cite{perthame2006transport}) allows the authors to prove the convergence of the normalised solutions to some stable distribution in some weighted $L^1$-space. We need stronger convergence in order to study the long-time behaviour of the masses of the solutions of (\ref{eq:edpmono}). Since the birth  matrix $\textbf{B}_x$ is not irreducible (it is triangular) and the parameters $\textbf{B}_x$ and $\textbf{D}_x$ are not smooth, we cannot apply Theorems 4.9 and 4.11 in \cite{webb1985theory}. Nonetheless, we easily extend them to our reducible and non-smooth setting. 

\medskip \noindent We define the Malthusian parameter $\lambda(x)$ associated with some trait $x$ as the unique solution of the equation
\begin{equation}\label{eq:malthus}
\int_{0}^{x_b\wedge x_d}e^{-\lambda(x) a}da=1.
\end{equation}
Proposition \ref{le:linearlong} below justifies this definition and shows that $\lambda(x)$ is the asymptotic growth rate of the dynamics defined by (\ref{eq:monolinear}). Its proof is presented in the Appendix.
\\Let us define for all $z\in\mathbb{C}$ the $2\times 2$ matrix
\begin{equation}\label{eq:defmathF}
\textbf{F}(z) = \int_{\R_+}\textbf{B}_x(a)\exp\left( -\int_{0}^{a}(\textbf{D}_x(\alpha)+ z \textbf{I})d\alpha\right)da.
\end{equation}
Note that it is well-defined since $\textbf{B}_x$ has compact support.
\begin{prop}\label{le:linearlong}
Assume that $x\in\mathcal{V}$. Then the linear operator $A$ admits a unique pair of simple principal eigenelements $(\lambda(x),N_x)\in \R_+^{*}\times D(A)$ where the stable age distribution $N_x$ satisfies
\begin{equation*}
N_x^1(a)=e^{-\left(\lambda(x) a + (a-x_d)\vee 0 \right)},\quad N_x^2(a)=\frac{\left[\textbf{F}(\lambda(x))\right]_{21}}{1-\left[\textbf{F}(\lambda(x))\right]_{22}} e^{-(1+\lambda(x))a}.
\end{equation*}
Moreover, for any  non-negative solution $v_x(t,a)$ of (\ref{eq:monolinear}) in $L^1(\R_+)^2$, there exists a positive constant $c(v_x(0,.))$ such that $e^{-\lambda(x) t}v_x(t,.)\rightarrow c(v_x(0,.))N_x$ in $L^{1}(\R_+)^2$ as $t\rightarrow +\infty$.
\end{prop}
Let us now give a lemma which will be used to study the long-time behaviour of the
masses of the solutions of (\ref{eq:edpmono}) and whose proof is postponed to the Appendix.
\begin{lemm}\label{le:edo}
Let $(m_{11},m_{12},m_{22})\in \R_+^{*}\times\R_+\times\R_{-}^{*}$. Let $\mathcal{D}_{11}(t),\mathcal{D}_{12}(t),\mathcal{D}_{22}(t)$ be continuous functions from $\R_+$ to $\R$ which tend to zero as $t\rightarrow \infty$. Let us denote
\begin{equation*}
M=\left(\begin{array}{cc}
m_{11} & 0 \\ 
m_{21} & m_{22}
\end{array} \right),\quad \mathcal{D}(t)=\left(\begin{array}{cc}
\mathcal{D}_{11}(t) & 0 \\ 
\mathcal{D}_{21}(t) & \mathcal{D}_{22}(t)
\end{array} \right).
\end{equation*}
Then any solution $(z(t),t\geq 0)$ of the equation
\begin{equation}\label{eq:ode}
\frac{\text{d}z(t)}{\text{dt}}=(M +\mathcal{D}(t))z(t) - \eta\Vert z(t)\Vert_1\text{ } z(t)
\end{equation}
started at $z(0)\in \R_+^{*}\times \R_+$ converges to a vector $\overline{z}$ which satisfies
\begin{equation}\label{eq:statioedo}
\overline{z}_1= \frac{m_{11}}{\eta}\frac{1}{1+\frac{m_{21}}{m_{11}-m_{22}}},\quad \overline{z}_2 = \frac{m_{21}}{m_{11}-m_{22}}\overline{z}_1.
\end{equation}
\end{lemm} 
We conclude this section by proving Proposition \ref{prop:monoedp}.
\begin{proof}[of Proposition \ref{prop:monoedp}]
We prove the first assertion. Let $x\in\mathcal{V}$ and let $\lambda(x)$ be the principal eigenvalue of $A$ given by Proposition \ref{le:linearlong}. Let $\overline{n}_x$ be the (unique) principal eigenvector of $A$ which satisfies $\eta\,\Vert \overline{n}_x\Vert_1 =\lambda(x)$. It is obvious that $\overline{n}_x$ is a non-trivial stationary state of (\ref{eq:edpmono}). Reciprocally, let $\overline{n}$ be a stationary state of (\ref{eq:edpmono}). Then we have necessarily  $\lambda(x) =\eta\,\Vert \overline{n}\Vert_1$ and that $\overline{n}$ is an eigenvector of $A$ associated with the eigenvalue $\lambda(x)$ that allows us to conclude. We now study the long-time behaviour of the solutions. Let us define 
\begin{equation*}
v_x(t,a)=\exp\left(\eta\int_{0}^{t}\Vert n_x(s,.)\Vert_1 ds\right)n_x(t,a).
\end{equation*}
It is straightforward to prove that $v_x$ is a solution of the linear equation (\ref{eq:monolinear}). By Proposition \ref{le:linearlong} we have $e^{-\lambda(x) t}v_x(t,.)\rightarrow c(v_x(0,.))N_x$ in $L^1(\R_+)^2$ as $t\rightarrow\infty$. We deduce that for $i\in\lbrace 1,2\rbrace$ and denoting $\rho_x^i(t)=\Vert n_x^i(t,.)\Vert_1$,
\begin{equation}\label{eq:convnormsol}
\frac{n_x^{i}(t,.)}{\rho_x^i(t)}=\frac{e^{-\lambda(x) t}v_x^{i}(t,.)}{\int_{\R_+}e^{-\lambda(x) t}v_x^{i}(t,\alpha)d\alpha}\rightarrow \frac{N^{i}_x}{\int_{\R_+}N_x^i(\alpha)d\alpha}
\end{equation}
in $L^1(\R_+)$ as $t\rightarrow +\infty$. We now study the behaviour of the masses $\rho_x(t)$. By taking the derivative under the integral, we obtain that for $i\in\lbrace 1,2\rbrace$
\begin{align*}
\frac{\text{d}\rho_x^i(t)}{\text{dt}}&=\sum_{j=1}^{2}\int_{\R_+}(\left[\textbf{B}_x(\alpha)\right]_{ij} -\left[\textbf{D}_x(\alpha)\right]_{ij})n_x^j(t,\alpha)d\alpha - \eta \rho_x^i(t)\Vert \rho_x(t)\Vert_1\\
&= \sum_{j=1}^{2}\rho_x^j(t)\int_{\R_+}(\left[\textbf{B}_x(\alpha)\right]_{ij} -\left[\textbf{D}_x(\alpha)\right]_{ij})\frac{n_x^j(t,\alpha)}{\rho_x^j(t)}d\alpha -\eta \rho_x^i(t)\Vert \rho_x(t)\Vert_1.
\end{align*}
Hence we obtain by (\ref{eq:convnormsol}) that
\begin{equation*}
\frac{\text{d}\rho_x(t)}{\text{dt}}=(\textbf{A} +\mathcal{A}(t))\rho_x(t) - \eta\Vert \rho_x(t)\Vert_1\rho_x(t)
\end{equation*}
where $\textbf{A}=(a_{ij})$ and for $(i,j)\in \lbrace 1,2\rbrace^2$,
\begin{equation}\label{eq:a}
a_{ij}= \int_{\R_+}(\left[\textbf{B}_x(\alpha)\right]_{ij} -\left[\textbf{D}_x(\alpha)\right]_{ij})\frac{N^{j}_x(\alpha)}{\int_{\R_+}N^{j}_x}d\alpha,
\end{equation}
and $\mathcal{A}(t)$ is a continuous function decreasing to zero as t tends to infinity. Since $a_{11}=\lambda(x) >0$, $a_{21}\geq 0$, $a_{12}=0$ and $a_{22}<0$, Lemma \ref{le:edo} allows us to conclude that $\rho_x(t)$ converges to $\overline{\rho}_x$, which is defined as the unique  solution of the equation $A\overline{\rho}_x - \eta\Vert \overline{\rho}_x\Vert_1 \overline{\rho}_x =0$. We easily solve this system and we obtain that $\overline{\rho}_x$ satisfies 
\begin{equation}\label{eq:satiomass}
\overline{\rho}^1_x=\frac{\lambda(x)}{\eta}\frac{1}{1 + \frac{a_{21}}{\lambda(x) -a_{22}}},\quad \overline{\rho}^2_x = \frac{a_{21}}{\lambda(x) -a_{22}}\overline{\rho}^1_x.
\end{equation}
\end{proof}
We conclude this section by writing more explicit formulas for the stationary state $\overline{n}_x$.
\begin{prop}\label{pr:explicitstatio}
Let $x\in\mathcal{V}$ then we have
\begin{equation*}
\overline{n}^1_x(a)=\overline{\rho}^1_x\frac{N_x^1(a)}{\int_{\R_+} N_x^1(\alpha)d\alpha},\quad \overline{n}^2_x(a)=\overline{\rho}^2_x \frac{N_x^2(a)}{\int_{\R_+} N_x^2(\alpha)d\alpha}
\end{equation*}
where
\begin{equation*}
\overline{\rho}^1_x=\frac{\lambda(x)}{\eta}\frac{1}{1 + \frac{a_{21}}{\lambda(x) -a_{22}}},\quad \overline{\rho}^2_x = \frac{a_{21}}{\lambda(x) -a_{22}}\overline{\rho}^1_x,
\end{equation*}
$N_x$ is defined in Proposition \ref{le:linearlong} and $a_{ij}$ are defined in (\ref{eq:a}).
\end{prop}
\begin{proof}
It is a direct consequence of (\ref{eq:convnormsol}) and (\ref{eq:satiomass}).
\end{proof}
\begin{biol}
Equation (\ref{eq:edpmono}) describes the dynamics of a large monomorphic population with trait $x$. Proposition \ref{prop:monoedp} shows that the age distribution of the population stabilizes around the equilibrium $\overline{n}_x =(\overline{n}_x^1,\overline{n}_x^2)$. The equilibria $\overline{n}^1_x$ and $\overline{n}_x^2$ describe the age equilibria of the population with trait $(x_b,x_d)$ and $(x_b,0)$ respectively. We observe that if $x_b<x_d$, then  $a_{21}=0= \overline{\rho}^2_x$ and  Proposition \ref{pr:explicitstatio} leads to the equilibrium  $\overline{n}_x=(\overline{n}_x^1,0)$. 
\end{biol}
\subsection{Bimorphic dynamics}
Let $x=(x_b,x_d)\in\mathcal{V}$ and $ y=(y_b,y_d)\in\mathcal{V}$. We consider a bimorphic initial population such that the sequence $(Z_0^K)_K$ weakly converges to $\delta_x n_x(a)da +\delta_y n_y(a)da$ as $K$ tends to infinity. Using similar arguments as in \cite{tran2008large}, we can prove that the sequence of processes $(Z^K)_K$ converges in probability,  on any finite time interval, to the solution $((n_x(t,.),n_y(t,.)),t\geq 0)\in L^1(\R_+)^4$ of the following system of non-linear partial differential equations
\begin{equation}\label{eq:bimo}
\begin{cases}
\partial_t n_x(t,a) +\partial_a n_x(t,a) = -\left( \textbf{D}_x(a) +\eta (\Vert n_x(t,.)\Vert_1 +\Vert n_y(t,.)\Vert_1)\textbf{I}\right)n_x(t,a)\\
n_x(t,0)=\int_{\R_+}\textbf{B}_x(\alpha)n_x(t,\alpha)d\alpha\\
\partial_t n_y(t,a) +\partial_a n_y(t,a) = -\left( \textbf{D}_y(a) +\eta (\Vert n_x(t,.)\Vert_1 +\Vert n_y(t,.)\Vert_1)\textbf{I}\right)n_y(t,a)\\
n_y(t,0)=\int_{\R_+}\textbf{B}_y(\alpha)n_y(t,\alpha)d\alpha,
\end{cases}
\end{equation}
with initial condition given by $(n_x(0,a),n_y(0,a))=(n_x(a),n_y(a))$. Equations (\ref{eq:bimo}) describe the dynamics of a dimorphic population with traits $x$ and $y$ interacting through competition. We prove the following proposition.
\begin{prop}\label{pr:bimoedp}
Let $x,y\in\mathcal{V}$ such that $x_b\wedge x_d< y_b\wedge y_d$.  Then any $L^1(\R_+)^4$-solution $(n_x(t,.),n_y(t,.))$ of (\ref{eq:bimo})  satisfying $n_y^1(t,.)\neq 0$ converges to $(0,\overline{n}_y)$ in $L^1(\R_+)^4$ as $t\rightarrow\infty$.
\end{prop}
\begin{proof}
Let $x,y\in\mathcal{V}$ such that $x_b\wedge x_d < y_b\wedge y_d$. Then by \eqref{eq:malthus}, we have $\lambda(x)<\lambda(y)$. For $u\in \lbrace x,y\rbrace$ we define
\begin{equation*}
v_u(t,a)=\exp\left(\eta\int_{0}^{t}(\Vert n_x(s,.)\Vert_1 +\Vert n_y(s,.)\Vert_1)ds\right)n_u(t,a).
\end{equation*}
The functions $v_x$ and $v_y$ are solutions of the linear systems (\ref{eq:monolinear}). We deduce from Proposition \ref{le:linearlong} that $e^{-\lambda(u) t}v_u(t,.)\rightarrow c(v_u(0,.))N_u$ in $L^1(\R_+)^2$ as $t\rightarrow  + \infty$, for a positive constant $c(v_u(0,.))$. Hence, we obtain that $e^{-\lambda(x) t}\int_{\R_+}v_x^1(t,\alpha)d\alpha$ converges to  a positive limit. Since $\lambda(y)>\lambda(x)$ it comes that $e^{-\lambda(y) t}\int_{\R_+}v_x^1(t,\alpha)d\alpha$ converges to $0$ as $t\rightarrow+\infty$. We deduce that
\begin{equation*}
\frac{\rho_y^1(t)}{\rho_x^1(t)}=\frac{e^{-\lambda(y) t}\int_{\R_+}v_y^1(t,\alpha)d\alpha}{e^{-\lambda(y) t}\int_{\R_+}v_x^1(t,\alpha)d\alpha}\rightarrow +\infty.
\end{equation*}
Since $\rho_y^1(t)$ is bounded we deduce that $\rho_x^1(t)\rightarrow 0$ and similarly that $\rho_x^2(t)\rightarrow 0$. Then the population with trait $x$ becomes extinct. Using similar arguments as in the previous proof we obtain that $n_y(t,.)\rightarrow \overline{n}_y$ in $L^1(\R_+)^2$ which allows us to conclude.
\end{proof}
\begin{biol}
Equation (\ref{eq:bimo}) describes a competition dynamics between two infinite monomorphic populations with trait $x$ and $y$. Proposition \ref{pr:bimoedp} shows that if $x_b\wedge x_d < y_b\wedge y_d$, then the population with trait $y$ invades and becomes fixed while  the population with trait $x$ becomes extinct. That  gives us an invasion-implies-fixation criterion.
\end{biol}
\section{Adaptive dynamics analysis}
In this section, we study the model introduced in Section 2 under the different scaling of the adaptive dynamics. We generalise the Trait Substitution Sequence with age structure (\cite{meleard2009trait}) to take into account the Lansing Effect. Then we study the behaviour of the TSS on a large time-scale when mutations are small. We show that the limiting behaviour of the TSS is described by  a differential inclusion which generalises the canonical equation of adaptive dynamics (\cite{dieckmann1996dynamical}, \cite{champagnat2002canonical}, \cite{champagnat2011polymorphic}) to non-regular fitness  functions. We first state some properties of the demographic parameters and  introduce the invasion fitness function.
\subsection{Malthusian parameter and invasion fitness}
Let us introduce the following sets: $U_1=\lbrace x\in\mathcal{V}:x_b<x_d\rbrace$, $U_2=\lbrace x\in\mathcal{V}:x_d<x_b\rbrace$ and $\mathcal{H}=\lbrace x\in\mathcal{V}:x_b=x_d\rbrace$.
\begin{figure}[h!]
\begin{center}
\includegraphics[scale=0.2]{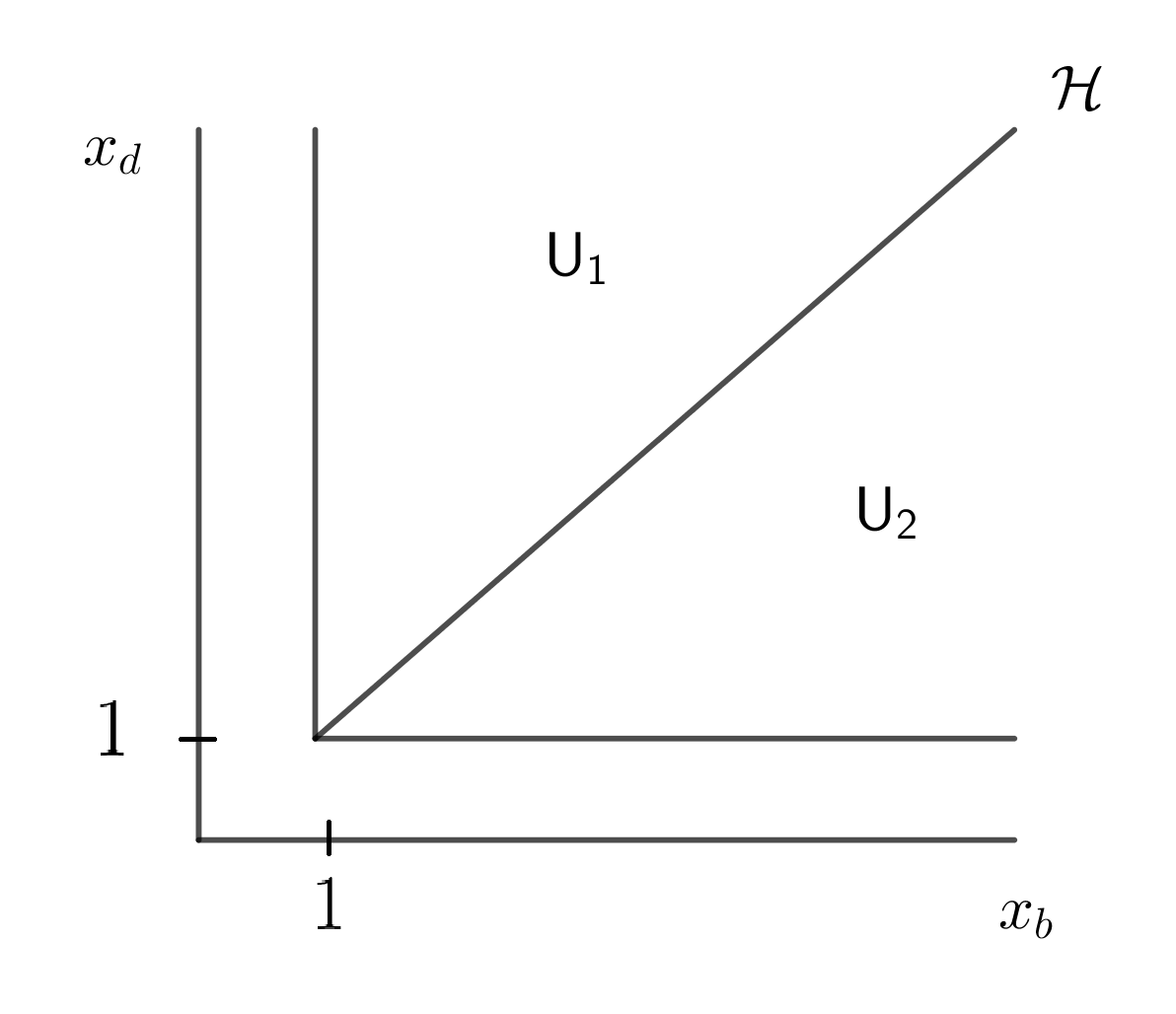}
\caption{Picture of $\mathcal{V}=U_1\cup \mathcal{H}\cup U_2$.}
\label{fig:dessinV}
\end{center}
\end{figure}
\subsubsection{Malthusian parameter}
We now give some properties of the Malthusian parameter $\lambda(x)$ defined in (\ref{eq:malthus}). 
\begin{prop}\label{pr:proprietesmalthus}
\begin{itemize}
\item[(i)]For all $x\in\mathcal{V}$,  $0\leq \lambda(x)<1$.
\item[(ii)] The map $x\in\mathcal{V}\mapsto \lambda(x)$ is continuous. It is differentiable on $U_1\cup U_2$ and satisfies
\begin{align}\label{eq:gradmalthus}
&\forall x\in U_1,\quad \nabla \lambda(x)= \left(\frac{e^{-\lambda(x) x_b}}{G(x)}, 0\right)\\\nonumber
&\forall x\in U_2,\quad \nabla \lambda(x)= \left(0,\frac{e^{-\lambda(x)x_d}}{G(x)}\right)
\end{align}
where $G(x)=\int_{0}^{x_b\wedge x_d}a e^{-\lambda(x) a}da$. 
\item[(iii)]We have $\sup_{x\in U_1\cup U_2}\Vert \nabla \lambda(x)\Vert < +\infty$. Moreover, for all $i\in\lbrace 1,2\rbrace$, the fitness gradient $\nabla \lambda$ is Lipschitz on $U_i$.
\end{itemize}
\end{prop}
\begin{rema}
Note that the Malthusian parameter $\lambda(x)$ is not differentiable on the diagonal $\mathcal{H}$, which can be easily obtained by computing left and right partial derivatives on $\mathcal{H}$.
\end{rema}
\begin{nota}
For all $h\in \R$ we define $(h)_1 = \left(\begin{array}{c}
h \\ 
0
\end{array}\right) $ and $(h)_2=\left(\begin{array}{c}
0 \\ 
h
\end{array}\right)$.
\end{nota}
\begin{proof}
(i) Since $x_b\wedge x_d >1$ we obtain from the definition (\ref{eq:malthus}) that $\lambda(x)> 0$. Assume that $\lambda(x)\geq 1$. Then we obtain that $1=\int_{0}^{x_b\wedge x_d}e^{-\lambda(x) a}da \leq 1 - e^{-x_b\wedge x_d}$ which is a contradiction.
\\
(ii) We prove the continuity. Let $x\in\mathcal{V}$ and let $(x_n)$ be a sequence of $\mathcal{V}$ such that $x^n\rightarrow x$. By (i), $\lambda$ is bounded and we can extract a subsequence (still denoted $x^n$ by simplicity)  such that $\lambda(x^n)\rightarrow \lambda^{*}$. We deduce that $1 = \int_{0}^{x^n_b\wedge x^n_d}e^{-\lambda(x^n)a}da\rightarrow  \int_{0}^{x_b\wedge x_d}e^{-\lambda^{*}a}da=1$ which allows us to conclude. Differentiability properties are a direct consequence of the Implicit Function Theorem. For each $i\in\lbrace 1,2\rbrace$ we apply implicit function Theorem to the map
\begin{equation*}
(\lambda,x)\in \R\times U_i\mapsto \int_{0}^{x_b\wedge x_d}e^{-\lambda a}da - 1.
\end{equation*}
We deduce that $\lambda$ is differentiable over $U_1\cup U_2$ and that
\begin{equation*}
\forall i\in\lbrace 1,2\rbrace,\quad \forall x \in U_i,\quad \nabla \lambda(x)= \left(\frac{e^{-\lambda(x) (x_b\wedge x_d)}}{G(x)}\right)_i.
\end{equation*}
(iii) It is straightforward to check that for all $x\in\mathcal{V}$,
\begin{equation*}
\frac{e^{-\lambda(x)(x_b\wedge x_d)}}{G(x)}\leq \frac{1}{\int_{0}^{1}a e^{-a}da}
\end{equation*}
which allows us to obtain that $\sup_{x\in U_1\cup U_2}\Vert \nabla \lambda(x)\Vert < +\infty$. Moreover, the gradient $\nabla \lambda$ is obviously differentiable on $U_i$. Since $G$ is bounded below by $\int_{0}^{1}ae^{-a}da$, we deduce that $\nabla \lambda$ has bounded derivatives on $U_i$ and that $\nabla \lambda$ is Lipschitz on $U_i$.
\end{proof}
\begin{rema}\label{remacaswell2}
Formulae (\ref{eq:gradmalthus}) describe the sensitivity of the Malthusian parameter to small variations of the trait $x$ as well as the strength of selection at ages $x_b$ and $x_d$ for a population with Lansing effect. The quantity $G(x)$ can be interpreted as the mean generation time associated with the trait $x$. Moreover (\ref{eq:gradmalthus}) and Proposition \ref{le:linearlong} yield
\begin{equation}\label{eq:gradfitstable}
\forall i\in\lbrace 1,2\rbrace,\quad \forall x \in U_i,\quad \nabla \lambda(x)= \left(\frac{N_x^1(x_b\wedge x_d)}{G(x)}\right)_i
\end{equation}
where $N_x^1$ is the stable age distribution. In \cite{caswell2010reproductive}, Caswell obtains similar formulae for derivatives of the Malthusian parameter with respect to some little perturbations on the intensity of birth or death at some given age while our formulae are obtained considering a small perturbation on the duration of the reproduction phase (not on the intensity which remains constant equal to one).
\end{rema}
The following proposition recalls a simple link between the Malthusian parameter and the stationary state of the monomorphic partial differential equation (\ref{eq:edpmono}).
\begin{prop}
\begin{itemize}
\item[(i)] For all $x\in\mathcal{V}$, we have $\lambda(x) = \eta\,\Vert \overline{n}_x\Vert_1$. \item[(ii)] The map $x\in \mathcal{V}\mapsto \overline{n}_x(0)$ is continuous and bounded.
\end{itemize}
\end{prop}
\begin{proof}
(i)  has been proved at the beginning of the proof of Proposition \ref{prop:monoedp}.
\\
(ii) By (i) we have 
\begin{equation}\label{eq:prop4.51}
\lambda(x) =\eta\left( \overline{n}_x^1(0)u_1(x) +\overline{n}_x^2(0)u_2(x)\right)
\end{equation}
where
\begin{align*}
u_1(x)&=\int_{0}^{+\infty}\exp\left(-\int_0^a\mathds{1}_{\alpha>x_d}d\alpha -\lambda(x) a\right)da\ ;\ 
u_2(x)=\int_{0}^{+\infty}\exp\left(-(1+\lambda(x))a\right)da.
\end{align*}
Moreover $\overline{n}_x(0)$ is a solution of  
\begin{equation}\label{eq:prop4.52}
\overline{n}_x(0)=\textbf{F}(\lambda(x))\,\overline{n}_x(0)
\end{equation}
where $\textbf{F}$ is defined in (\ref{eq:defmathF}). From (\ref{eq:prop4.51}) and (\ref{eq:prop4.52}) we obtain by simple computation that
\begin{equation*}
\overline{n}_x^1(0)=\frac{\lambda(x)}{\eta}\frac{1}{u_1(x) + u_2(x)\frac{\left[\textbf{F}(\lambda(x))\right]_{21}}{1-\left[\textbf{F}(\lambda(x))\right]_{22}}}
\end{equation*}
which is a continuous function of $x$. Boundedness is obvious arguing that $\overline{n}_x^1(0)\leq \Vert \overline{n}_x\Vert_1$.
\end{proof}
\subsubsection{Invasion fitness}
We extend the definition of the invasion fitness for age-structured populations introduced in \cite[Section 3]{meleard2009trait} to take into account the Lansing effect. 
\begin{defi}
For all $y\in\R_+^{*}$ and $x\in\mathcal{V}$, the invasion fitness $1-z(y,x)$ is defined as the survival probability of a bi-type age structured branching process with birth rates and death rates defined in \eqref{eq:birthdeathrate}, respectively equal to
\begin{equation*}
\textbf{B}_y(a)\quad \text{ and }\quad \textbf{D}_y(a) +\eta\,\Vert \overline{n}_x\Vert \textbf{I}.
\end{equation*}
\end{defi}
The next proposition gives a precise and precious relation between the invasion fitness and the Malthusian parameter.
\begin{prop}\label{pr:invasionfitness}
Let $y\in\R_+^{*}$ and $x\in\mathcal{V}$, then the invasion fitness satisfies
\begin{equation}\label{eq:invfit}
1-z(y,x)=(\lambda(y) - \lambda(x))\vee 0.
\end{equation}
\end{prop}
\begin{proof}
Let $Z_t(da)=(Z^1_t(da), Z_t^2(da))$ be an age- structured branching process with birth rates $\textbf{B}_y(a)$ and death rates $\textbf{D}_y(a) + \eta\,\Vert \overline{n}_x\Vert \textbf{I}= \textbf{D}_y(a) +\lambda(x) \textbf{I}$. The process $Z_t$ becomes extinct if and only if the process $Z_t^1$ becomes extinct. Indeed, if $Z_0^1 = 0$, the process $Z_t^2$ evolves as a sub-critical branching process. The process $Z_t^1$ is an age structured branching process with birth rates and death rates respectively
\begin{equation*}
\left[\textbf{B}_y(a)\right]_{11} = \mathds{1}_{a\leq y_b\wedge y_d}\quad\text{ and } \left[\textbf{D}_y(a)\right]_{11}+\lambda(x) = \mathds{1}_{a>y_d}+\lambda(x).
\end{equation*}
We deduce that $z(y,x)$ equals the smallest solution of the equation $z = F(z)$ where
\begin{align*}
F(z)&=\int_{\R_+}e^{(z-1)\int_{0}^{a}\left[\textbf{B}_y(\alpha)\right]_{11}d\alpha}\left(\left[\textbf{D}_y(a)\right]_{11}+\lambda(x)\right) e^{-\int_{0}^{a}\left(\left[\textbf{D}_y(\alpha)\right]_{11}+\lambda(x)\right)d\alpha}da\\
& =\int_{0}^{y_b\wedge y_d}\lambda(x) e^{(z-1-\lambda(x))a}da + e^{(z-1)(y_b\wedge y_d)} \int_{y_b\wedge y_d}^{y_d}\lambda(x) e^{-\lambda(x) a}da\\
& +e^{(z-1)(y_b\wedge y_d)}(1+\lambda(x))\int_{y_d}^{+\infty}e^{-(a-y_d)-\lambda(x) a}da\\
&=1 - e^{(z-1-\lambda(x))(y_b\wedge y_d)}+(z-1)\int_{0}^{y_b\wedge y_d}e^{(z-1-\lambda(x))a}da\\
&+e^{(z-1)y_b\wedge y_d}(e^{-\lambda(x) (y_b\wedge y_d)} -e^{-\lambda(x) y_d})+e^{(z-1)(y_b\wedge y_d)} +e^{(z-1)(y_b\wedge y_d)} e^{-\lambda(x) y_d}\\
&= 1 +(z-1)\int_{0}^{y_b\wedge y_d} e^{(z-1-\lambda(x))a} da.
\end{align*}
We have obtained that the equation $z=F(z)$ is equivalent to 
\begin{equation}\label{eq:profix}
z-1 = (z-1)\int_{0}^{y_b\wedge y_d} e^{(z-1-\lambda(x))a}da.
\end{equation}
If $\lambda(y)>\lambda(x)$, Equation (\ref{eq:profix}) admits two solutions $z=1$ and $z=\lambda(x) -\lambda(y) +1<1$. If $\lambda(y)<\lambda(x)$, Equation (\ref{eq:profix}) admits a  unique solution $z=1$. That allows us to conclude the proof.
\end{proof}
\begin{rema}
It is interesting to note that in our model, the invasion fitness (which is a concept from adaptive dynamics theory) and the Malthusian parameter are connected thanks to the simple relation (\ref{eq:invfit}).
\end{rema}
The following proposition characterises the set of traits $y$ which can invade some given trait $x$.
\begin{prop}\label{pr:invasionset}
For all $x\in\mathcal{V}$, $y\in\R_+^2$, 
\begin{equation*}
1-z(y,x)>0 \Longleftrightarrow \lambda(y)>\lambda(x)\Longleftrightarrow  y_b\wedge y_d > x_b\wedge x_d.
\end{equation*}
\end{prop}
\begin{proof}
From the definition (\ref{eq:malthus}) of the Malthusian parameter, we easily deduce the following equivalences:
\begin{align*}
(\lambda(y)-\lambda(x))\vee 0>0 & \Longleftrightarrow \lambda(y)-\lambda(x)>0\\
&\Longleftrightarrow \int_{0}^{x_b\wedge x_d}e^{-\lambda(y)a}da <1 \\
&\Longleftrightarrow x_b\wedge x_d<y_b\wedge y_d, 
\end{align*}
which concludes the proof.
\end{proof}
\subsection{Trait Substitution Sequence with age structure}
We first generalise the definition of the  Trait Substitution Sequence (TSS) with age structure defined in \cite{meleard2009trait} to take into account the Lansing Effect. 
\begin{defi}\label{def:TSS}
We define the measure valued process
$(T_t(dx,da),t\geq 0)$ by
\begin{equation*}
T_t(dx,da)=\delta_{X(t)}(dx) \overline{n}_{X(t)}^1(a)da +\delta_{(X_b(t),0)}(dx) \overline{n}_{X(t)}^2(a)da
\end{equation*}
where $(X(t),t\geq 0)= ((X_b(t),X_d(t)),t\geq 0)$ is defined as the pure jump Markov process on $\mathcal{V}$ with infinitesimal generator $L$ defined for all measurable and bounded function $\varphi:\mathcal{V}\rightarrow \R$ and $x\in\mathcal{V}$ by
\begin{align}\label{eq:intensitytss}
L\varphi(x)=\int_{(\R_+)^2}(\varphi(x+h)-\varphi(x))(\lambda(x+h)-\lambda(x))\text{ }\overline{n}_x^1(0)\text{ }\mu(dh)
\end{align}
where
\begin{equation*}
\mu(dh)=\frac{k(h_b)dh_b\otimes \delta_0(dh_d) +\delta_0(dh_b)\otimes k(h_d)dh_d}{2},
\end{equation*}
$k$ being defined in (\ref{eq:k}). 
\\The process $X$ will be called the \textit{Trait Substitution Sequence}. \end{defi}
\begin{rema}
The process $(T_t,t\geq 0)$ describes the evolution of the phenotypic structure of the population at the mutational time-scale. At each time, and because of the Lansing effect, the population is composed of two sub-populations: the first one corresponds to viable individuals with trait $X(t)=(X_b(t), X_d(t))\in\mathcal{V}$ whose age distribution is given by $\overline{n}_{X(t)}^1(a)da$; the second one is composed of individuals generated through the Lansing effect, with trait $(X_b(t), 0)$ and age distribution $\overline{n}_{X(t)}^2(a)da$.
\end{rema}
\begin{rema}
Figure \ref{fig:trajtss} describes the behaviour of the process $(X(t),t\geq 0)$. 
\\Any trait $x\in \mathcal{H}$ is an absorbing state for the process $X$. Indeed, by Proposition \ref{pr:invasionset}, for all $\varphi:\mathcal{V}\rightarrow \R$ measurable and bounded, for all $x\in\mathcal{H}$, $L\varphi(x)=0$. That means that when the trait $(x_b,x_d)$ satisfies $x_b=x_d$, no mutation can invade. 
\\By definition of the measure $\mu(dh)$ the process evolves horizontally or vertically (which means that the two traits $x_b$ and $x_d$ do not mutate simultaneously).  By Proposition \ref{pr:invasionset}, we deduce easily that the process evolves from the left to the right on $U_1$ and from bottom to top on $U_2$. 
\\Since the jump rates are continuous and tend to zero on $\mathcal{H}$, the process slows down as it approaches $\mathcal{H}$. 
\begin{figure}[!h]
\begin{center}
\includegraphics[scale=0.2]{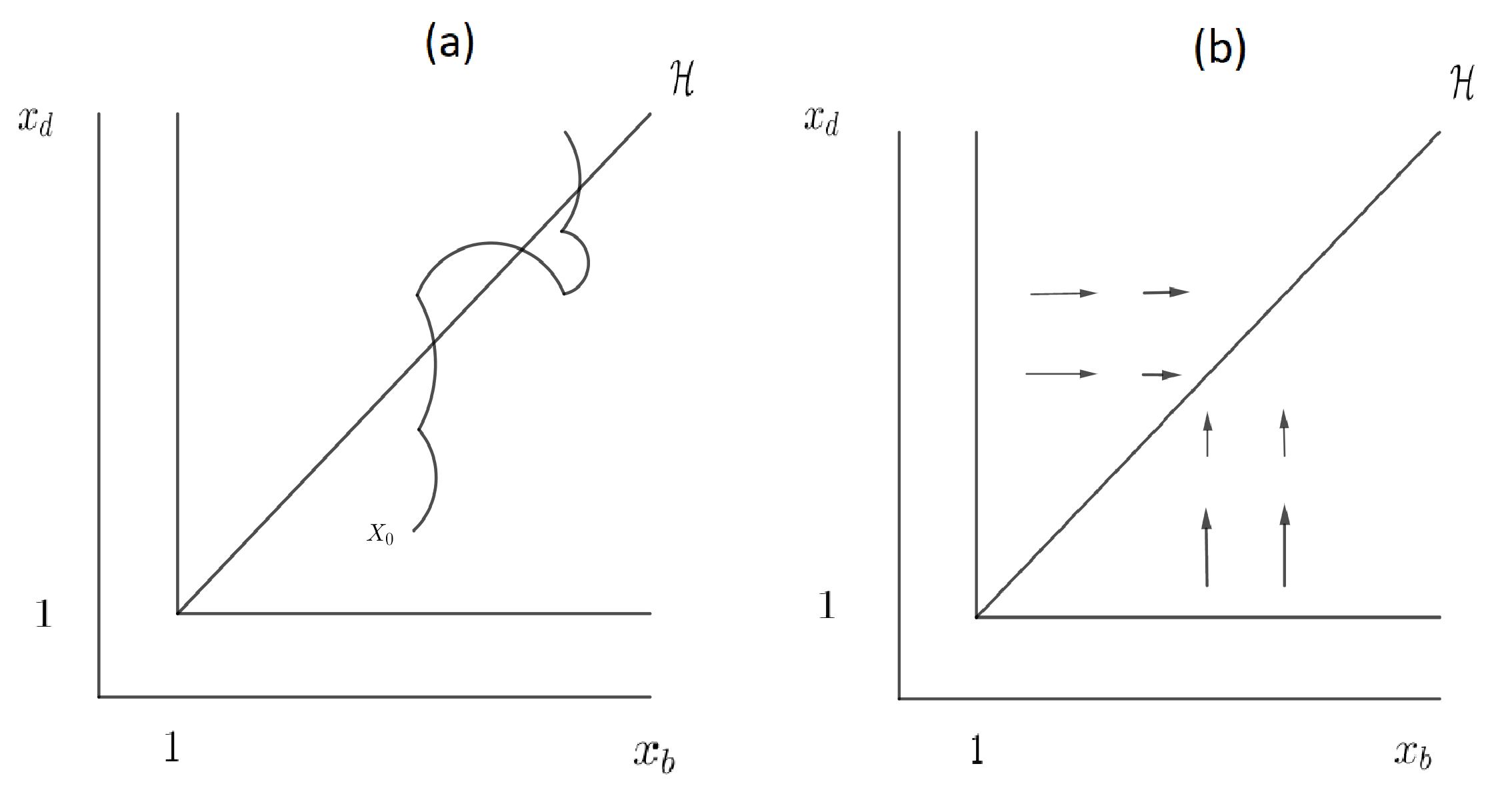}
\caption{(a): This picture represents a trajectory of the TSS process. (b): This picture represents the drift associated with the TSS process.}
\label{fig:trajtss}
\end{center}
\end{figure}
\end{rema}
We now explain the heuristics, rigorously proved in \cite{meleard2009trait}, which allow to obtain the TSS from the individual based model defined in Section 2.1. The main ideas  have been introduced in \cite{champagnat2006microscopic}, for a population without age-structure. They are based on the time-scale separation assumption on mutation probability $p_K$: as $K\rightarrow +\infty$
\begin{equation}\label{eq:scaleadaptive}
\forall V>0,\quad \exp(-KV)=o(p_K),\quad p_K=o\left(\frac{1}{K\log(K)}\right),
\end{equation}
which allows to separate the effect of the natural selection and the appearance of new mutants. Let $x\in \mathcal{V}$ and consider a sequence $(Z_0^K)_{K}$  converging to $ \delta_x n_x^1(0,a)da $ as $K\rightarrow +\infty$.

\vspace{0.3cm}
\textbf{1) Monomorphic approximation.} For large $K$, the process $Z_t^K$ stays close to the measure $\delta_x n_x^1(t,.) + \delta_{(x_b,0)} n_x^2(t,.)$ where $n_x(t,.)=(n_x^1(t,.),n_x^2(t,.))$ satisfies the partial differential equation (\ref{eq:edpmono}). By Proposition \ref{prop:monoedp}, the dynamics $n_x(t,.)$ converges to $\overline{n}_x= (\overline{n}_x^1,\overline{n}_x^2)$ as $t$ tends to infinity and hence reaches a given neighbourhood of $\overline{n}_x$ in finite time. By using large deviation results (\cite{tran2008large}), we obtain with probability tending to one as $K$ tends to infinity that the process $Z_t^K$ stays in this neighbourhood of $\delta_x \overline{n}_x^1 + \delta_{(x_b,0)} \overline{n}_x^2$ during a time $e^{CK}$ for some $C>0$. The left-hand side in Assumption (\ref{eq:scaleadaptive}) ensures that the next mutation appears before the process leaves this neighbourhood.

\vspace{0.3cm}
\textbf{2) Appearance of a mutant.}
We deduce that the monomorphic population with trait $x$ creates a mutant with trait: 
\begin{itemize}
\item[(i)] $y=(x_b +h_b,x_d)$ or $y=(x_b ,x_d +h_d)$ at a rate approximately equal to
$2K p_K(1-p_K)\overline{n}^{1}_{x}(0)$;
\item[(ii)]$y=(x_b +h_b,x_d+h_d)$ at a rate approximately equal to $Kp_K^2 \overline{n}_{x}^1(0)$;
\item[(iii)]$y=(x_b, h_d)$ or $y=(x_b+h_b, 0)$ at  at rate approximately equal to $2Kp_K(1-p_K)\overline{n}_x^2(0)$;
\item[(iv)]$y=(x_b +h_b, h_d)$ at rate approximately equal to $Kp_K^2 \overline{n}_x^2(0)$.
\end{itemize}
where the variables $h_b$ and $h_d$ are chosen independently with distribution $k$. 

\vspace{0.2cm}
Since $p_K^2=o(p_K(1-p_K))$, the cases (ii) and (iv) cannot be observed on the mutation time-scale $t/2K(1-p_K)$.

\vspace{0.3cm}
\textbf{3) Effect of the natural selection.}
In cases (i) and (iii), the mutant population dynamics is approximated by a bi-type age structured branching process with birth rates $\textbf{B}_y$ and death rates $\textbf{D}_y +\eta \Vert \overline{n}_x\Vert_1 \textbf{I}$. By Proposition \ref{pr:invasionfitness}, the mutant population survives with probability
\begin{equation*} 
(\lambda(y)-\lambda(x))\vee 0.
\end{equation*}
Let us detail the two different cases.
\begin{itemize}
\item[•] Case (i). With  probability $1/2$, the trait of the mutant is  $y=(x_b+h_b, x_d)$. From Proposition \ref{pr:invasionfitness}, we deduce that the mutant can survive if and only if
\begin{align}\label{eq:inv1}
\lambda(x_b+h_b,x_d)>\lambda(x)\Longleftrightarrow h_b>0\text{ and }x\in U_1.
\end{align}
With probability $1/2$, $y=(x_b,x_d+h_d)$ and can survive if and only if
\begin{align}\label{eq:inv2}
\lambda(x_b,x_d +h_d)>\lambda(x)\Longleftrightarrow h_d>0\text{ and }x\in U_2.
\end{align}
\item[•] Case (iii) (Lansing effect). The mutant has the trait $y=(x_b, h_d)$ or $y=(x_b +h_b, 0)$. By (\ref{eq:k}), we have $y_b\wedge y_d<1$ which implies that $\lambda(y)<0$ and then $(\lambda(y)-\lambda(x))\vee 0 =0$. In this case,  the mutant population becomes extinct.
\end{itemize}
We deduce that the mutant can only survive (with positive probability) in case (i). The birth rate of such a mutant (on the time-scale $t/2p_K(1-p_K)$) is given by the intensity measure on $\R$
\begin{equation*}
\overline{n}_x^1(0)\mu(dh)
\end{equation*}
that leads to the right hand side in (\ref{eq:intensitytss}). The probability that such a mutant survives and reaches a size of  order $K$ equals 
\begin{equation*}
(\lambda(y)-\lambda(x))\vee 0.
\end{equation*} 
Moreover (\ref{eq:inv1}) and (\ref{eq:inv2}) imply that 
\begin{equation*}
((\lambda(x+h)-\lambda(x))\vee 0)\mu(dh)= (\lambda(x+h)-\lambda(x))\mathds{1}_{\R_+^2}(h)\mu(dh).
\end{equation*}
This allows to obtain the left hand side of (\ref{eq:intensitytss}).
\\If the mutant population becomes extinct, the resident population remains close to its equilibrium $\overline{n}_x$.
\\If the mutant population survives, then it reaches a size of order $K$ with a probability that tends to one and the  population dynamics is approximated by the solution $(n_x(t,.),n_y(t,.))$ of the bimorphic system of partial differential equations (\ref{eq:bimo}). In this case we have necessarily $\lambda(y)>\lambda(x)$. Following Proposition \ref{pr:bimoedp}, the deterministic dynamics $(n_x(t,.),n_y(t,.))$  reaches a neighbourhood of $(0,\overline{n}_y)$. By using branching processes approximations and  arguments introduced in \cite{champagnat2006microscopic}, we can deduce that the resident population with trait $x$ becomes extinct. One can prove as in \cite{champagnat2006microscopic} that this competition phase has a duration of order $\log(K)$.
\\The right hand side of Assumption (\ref{eq:scaleadaptive}) ensures that the three steps of invasion are completed before the next mutation occurs. \begin{flushright}
•
\end{flushright}The Markov property allows to reiterate the same reasoning for the next mutation occurence.  In summary, the following theorem holds. 
\begin{theo}\label{th:approxtss}
The following convergence holds in the sense of finite dimensional marginals:
\begin{equation*}
\left(Z_{\frac{t}{2Kp_K(1-p_K)}}^K, t\geq 0\right)\longrightarrow (T_t, t\geq 0),\quad \text{ as }K\rightarrow \infty,
\end{equation*}
where the process $T$ is defined in Definition \ref{def:TSS}.
\end{theo}
\subsection{A canonical inclusion for adaptive dynamics}
In this section, we assume in that mutations are small. We study the behaviour of the process $X$ defined in (\ref{eq:intensitytss}) when the mutation scale equals $\epsilon>0$ and the time is rescaled by $1/\epsilon^2$. To this aim, we define the rescaled trait substitution sequence process $X^{\epsilon}$ and study the limiting behaviour of the process $X^{\epsilon}$ as $\epsilon \rightarrow 0$. In the usual cases (smooth fitness functions) the canonical equation introduced by Dieckmann-Law can be derived as limit of $X^{\epsilon}$ as $\epsilon\rightarrow 0$ (\cite{champagnat2011polymorphic}). As observed in Section 4.1, the fitness function $\lambda(x)$ does not satisfy these regularity assumptions. To overpass this difficulty, we use the approach developed in \cite{gast2012markov} based on differential inclusions. We prove in Theorem \ref{theo:diffinc} that the set of limit points of the family $X^{\epsilon}$ is characterised as the set of solutions of a differential inclusion.
\begin{defi}
The rescaled TSS process $(X^{\epsilon}(t),t\geq 0)$ is defined as a pure jump Markov process with infinitesimal generator $L^{\epsilon}$ defined for all measurable and bounded function $\varphi:\mathcal{V}\rightarrow \R$ and $x\in\mathcal{V}$ by
\begin{align}\label{eq:generatornormtss}
L^{\epsilon}\varphi(x)=\frac{1}{\epsilon^2}\int_{(\R_+)^2}(\varphi(x+\epsilon h)-\varphi(x))(\lambda(x+\epsilon h)-\lambda(x))\overline{n}_x^1(0)\mu(dh)
\end{align}
\end{defi}
\begin{rema}
The process $X^{\epsilon}$ shows a dynamics similar to the process $X$. The jump rates are of order $1/\epsilon$ and the jump sizes are of order $\epsilon$.
\end{rema}
We first introduce the set-valued map $F:\mathcal{V}\rightarrow \mathcal{P}(\R_+^2)$ defined for any $x\in\mathcal{V}$ by
\begin{align}\label{eq:mapvaluedset}
\forall i \in\lbrace 1,2\rbrace,\quad &\forall x\in U_i,\quad F(x)=\left( f(x,1) \right)_i \nonumber\\
&\forall x\in\mathcal{H},\quad F(x)=\left\lbrace \frac{1}{2}\left(\begin{array}{c}
f(x,u) \\ 
f(x,u)
\end{array}\right), u\in\left[0,1\right]\right\rbrace,
\end{align}
where for all $(x,u)\in\mathcal{V}\times\left[0,1\right]$,
\begin{equation}\label{eq:gradfit}
f(x,u)=\left(\int_{0}^{u}h^2k(h)dh +\int_{u}^{1} h u k(h)dh\right)\frac{e^{-\lambda(x) (x_b\wedge x_d)}}{G(x)}\frac{\overline{n}_x^1(0)}{2};
\end{equation}
and $G$ is defined in Proposition \ref{pr:proprietesmalthus}.

\vspace{0.2cm}
This set-valued map $F$ somehow generalises the classical fitness gradient. It is represented by a picture in Figure \ref{fig:dessinFH} (b). 
\\Let us explain the ideas leading to consider this function. Let us consider a compact subset $K$ of $U_i$. Since the Malthusian parameter $\lambda$ is differentiable on $U_i$, the following approximation holds: for all $h\in\R_+$, uniformly for $x\in K$, we have
\begin{equation}\label{eq:approx}
\lambda(x+\epsilon (h)_i)-\lambda(x)\approx \epsilon (h)_i.\nabla \lambda(x),
\end{equation}
which leads to the definition of $F$ on $U_i$. We analyse the case $x\in\mathcal{H}$ for which the approximation (\ref{eq:approx}) is not true. Indeed, let $x\in\mathcal{H}$ and let $u\in\left[0,1\right]$ and let us consider a sequence $x^{\epsilon}=x-\epsilon(u)_i$, we have 
\begin{align*}
\lambda(x^{\epsilon} +\epsilon (h)_i)-\lambda(x^{\epsilon})&=\lambda(x +\epsilon(h-u)_i) -\lambda(x-\epsilon(u)_i)\\
&=\lambda(x +\epsilon(h-u)_i)  -\lambda(x) +\lambda(x)-\lambda(x-\epsilon(u)_i)
\end{align*}
Assume $h<u$, since $\lambda$ is differentiable on $U_i$ we obtain when $\epsilon$ tends to $0$ that
\begin{align*}
\lambda(x^{\epsilon} +\epsilon (h)_i)-\lambda(x^{\epsilon}) \approx \epsilon (h-u)_i.\nabla \lambda(x) + \epsilon(u)_i.\nabla \lambda(x) \approx \epsilon (h)_i.\nabla \lambda(x),
\end{align*}
where $\nabla \lambda(x)$ is defined as the limit of $\nabla \lambda(y)$, $y\rightarrow x$, $y\in U_i$. That leads to the first integral in (\ref{eq:gradfit}).
If $h>u$, we obtain that  $\lambda(x +\epsilon(h-u)_i)  -\lambda(x)=0$ and 
\begin{align*}
\lambda(x^{\epsilon} +\epsilon (h)_i)-\lambda(x^{\epsilon})\approx \epsilon(u)_i.\nabla \lambda(x)\leq \epsilon(h)_i.\nabla \lambda(x),
\end{align*}
which leads to the second integral in (\ref{eq:gradfit}).
The inequality above means that when the process evolves near the diagonal $\mathcal{H}$ the adaptation slows down. Let us now define what we do  mean by a solution of the differential inclusion driven by the set-valued map $F$.
\begin{defi}
A solution of the differential inclusion driven by the set-valued map $F$ is an absolutely continuous function $x :\left[0,T\right]\rightarrow \mathcal{V}$ which satisfies $x(0)=c^0$ and for almost all $t\in\left[0,T\right]$,
\begin{equation}\label{eq:difinc}
\frac{\text{d}x(t)}{\text{dt}}\in F(x(t)).
\end{equation}
\end{defi}
It generalises the classical canonical equation for adaptive dynamics. For any $T>0$ and $x^0\in\mathcal{V}$, we denote by $\mathcal{S}_F(T,x^{0})$ the set of solutions of the differential inclusion (\ref{eq:difinc}). The following theorem characterises the limit of the process $X^{\epsilon}$ as the solution of the differential inclusion (\ref{eq:difinc}).
\begin{theo}\label{theo:diffinc}
Let $x^{0}\in\mathcal{V}$. Assume that $X^{\epsilon}(0)\rightarrow x^{0}$ in probability as $\epsilon\rightarrow 0$. For all $T,\delta>0$,
\begin{equation*}
\lim_{\epsilon\rightarrow 0}\mathbb{P}\left( \inf_{x\in \mathcal{S}_F(T,x^{0})}\sup_{t\in\left[0,T\right]}\vert X^{\epsilon}(t) -x(t)\vert >\delta\right) = 0.
\end{equation*}
\end{theo}
\begin{rema}
Theorem \ref{theo:diffinc} justifies our complete study. Let $(x(t),t\in\left[0,T\right])$ be a solution of (\ref{eq:difinc}). On each $U_i$, it satisfies
\begin{equation*}
\frac{\text{d}x(t)}{\text{d}t}=(f(x(t),1))_i,\quad x(t)\in U_i.
\end{equation*}
The map $x\in U_i\mapsto (f(x,1))_i$ is Lipschitz and bounded below by a positive constant. Hence, uniqueness holds on $U_i$ for (\ref{eq:difinc}) and any solution  reaches in finite time the diagonal $\mathcal{H}$. On $\mathcal{H}$, the solution satisfies
\begin{equation*}
\frac{\text{d}x(t)}{\text{d}t}\in F(x(t)),\quad x(t)\in \mathcal{H}.
\end{equation*}
Since for all $x\in\mathcal{H}$, $F(x)\subset \mathcal{H}$, we deduce that any solution stays in $\mathcal{H}$. 
\\Figure \ref{fig:simucanon} illustrates Theorem \ref{theo:diffinc}. We represent some trajectories of the process $(X^{\epsilon}(t),t\geq 0)$ started at $X^{\epsilon}(0)=(2, 1.5)$ for $\epsilon=0.001$. 
\begin{figure}[!h]
\begin{center}
\includegraphics[scale=0.42]{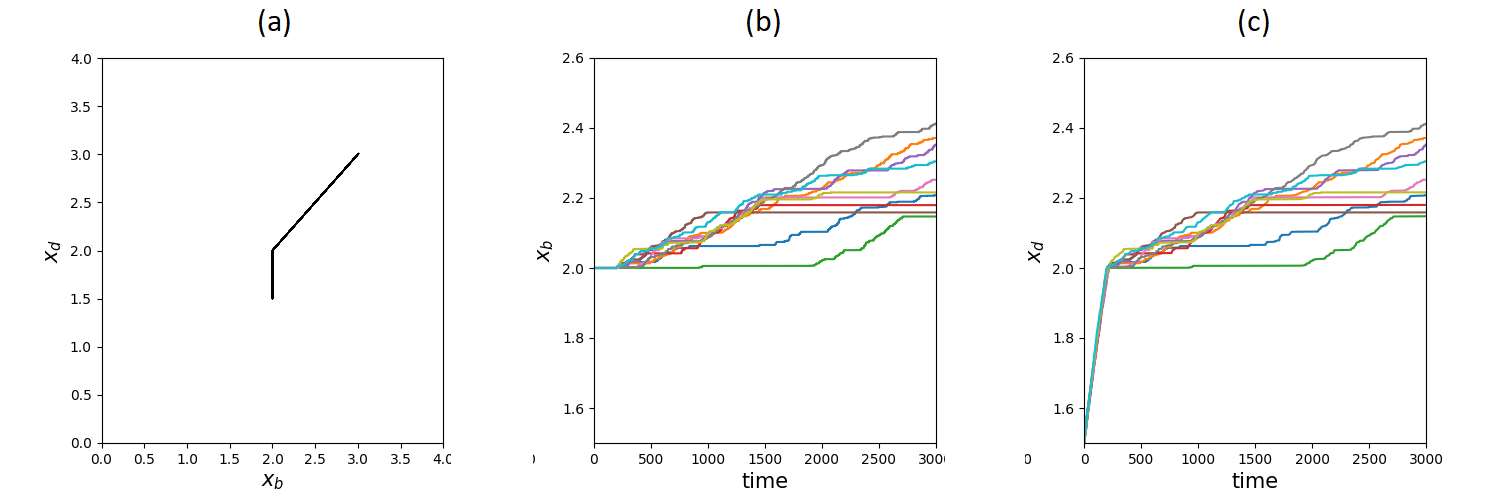}
\caption{We represent 10 simulations of the process $(X^{\epsilon}(t),t\geq 0)$ for $\epsilon = 0.001$, $X^{\epsilon}(0)=(2,1.5)$. (a): We represent $X^{\epsilon}_d(t)$ as a function of $X^{\epsilon}_b(t)$. (b): $(X^{\epsilon}_b(t),t\geq 0)$. (c): $(X^{\epsilon}_d(t),t\geq 0)$. We observe on (b) and (c) that before reaching the diagonal $\mathcal{H}$ the limit behaviour is deterministic. On $\mathcal{H}$, the process evolves with  speed in $F(x)$ for $x\in\mathcal{H}$.}
\label{fig:simucanon}
\end{center}
\end{figure}
\end{rema}
The proof of Theorem \ref{theo:diffinc} is based on \cite[Theorem 1]{gast2012markov}  recalled in the Appendix, see Theorem \ref{th:gastgaujal}. We start by writing the process $X^{\epsilon}$ as a time-changed Markov chain. We first re-write the generator $L^\epsilon$ as follows
\begin{align*}
L^{\epsilon}\varphi(x) = \int_{\R_+^2}(\varphi(x+h)-\varphi(x))\frac{\overline{n}_x^1(0)}{\epsilon}k^{\epsilon}(x,dh)
\end{align*}
where $k^{\epsilon}(x,dh)=$
\begin{equation*}
\frac{\lambda(x+h)-\lambda(x)}{\epsilon}\mu^{\epsilon}(dh) +\left(\int_{(\R_+)^2}\left(1-\frac{\lambda(x+g)-\lambda(x)}{\epsilon}\right)\mu^{\epsilon}(dg)\right)\delta_0(dh)
\end{equation*}
and $\mu^{\epsilon}$ is the image measure of  $\mu$ by the map $h\mapsto \epsilon h$. The following lemma is proved in \cite[Ch. 4, S. 2, p. 163]{ethier2009markov}.
\begin{lemm}[\cite{ethier2009markov}]
Let $\tau = \sup_{x\in\mathcal{V}}\overline{n}_x^1(0)$. Let $(Y^{\epsilon}(k),k\geq 0)$ be a Markov chain with jump law
\begin{equation*}
\tilde{k}^{\epsilon}(x,dh)=\frac{\overline{n}_x^1(0)}{\tau}k^{\epsilon}(x,dh) + \left(1-\frac{\overline{n}_x^1(0)}{\tau}\right)\delta_0(dh)
\end{equation*}
and $\Lambda^{\epsilon}$ be a Poisson process with intensity $\tau/\epsilon$. Then the processes $X^{\epsilon}$ and $Y^{\epsilon}(\Lambda^{\epsilon})$ have the same law.
\end{lemm}
Then we are led to study the Markov chain $(Y^{\epsilon}(k),k\geq 0)$. We first define the drift of the Markov chain $Y^{\epsilon}$ by 
\begin{equation*}
g_\epsilon(x)=\mathbb{E}\left[Y^{\epsilon}(1)-Y^{\epsilon}(0)\vert Y^{\epsilon}(0)=x\right].
\end{equation*}
A simple calculation gives us 
\begin{align*}
\forall i \in \lbrace 1,2\rbrace,\quad &\forall x\in U_i,\quad g_{\epsilon}(x)=\frac{\epsilon \overline{n}_x^1(0)}{2\tau}\int_{\R_+}\frac{\lambda(x+\epsilon(h)_i)-\lambda(x)}{\epsilon}k(h) (h)_i dh \nonumber\\
&\forall x\in\mathcal{H},\quad g_\epsilon(x)=0.
\end{align*} 
Then, we write the Markov chain $Y^{\epsilon}$ as a stochastic approximation algorithm
\begin{equation*}
Y^{\epsilon}(k+1)=Y^{\epsilon}(k) + \frac{\epsilon}{\tau}U^{\epsilon}(k) + g^{\epsilon}(Y^{\epsilon}(k))
\end{equation*}
where $U^{\epsilon}$ is a martingale difference sequence. Assumptions of Theorem \ref{th:gastgaujal}  are clearly satisfied. In order to apply it, we compute the following set-valued map
\begin{equation*}
\forall x\in\mathcal{V},\quad H(x)=\text{conv}\left\lbrace \text{acc}_{\epsilon\rightarrow 0}\frac{\tau g_{\epsilon}(x^{\epsilon})}{\epsilon}:x^{\epsilon}\rightarrow x\right\rbrace
\end{equation*}
where $\text{conv}(A)$ denotes the smallest convex set which contains $A$ and $\text{acc}_{\epsilon\rightarrow 0}\tau g_{\epsilon}(x^{\epsilon})/\epsilon$ is the set of accumulation points of the sequence $\tau g_{\epsilon}(x^{\epsilon})/\epsilon$ as $\epsilon$ tends to zero. 
\begin{lemm}\label{le:mapset}
The set-valued map $H$ satisfies
\begin{align*}
\forall i \in\lbrace 1,2\rbrace,\quad &\forall x\in U_i,\quad H(x)=\lbrace\left( f(x,1) \right)_i\rbrace\\
&\forall x\in\mathcal{H},\quad H(x)=\left\lbrace \alpha\left(\begin{array}{c}
f(x,u) \\ 
0
\end{array}\right)+(1-\alpha)\left(\begin{array}{c}
0 \\ 
f(x,v)
\end{array}\right):(u,v,\alpha)\in\left[0,1\right]^3\right\rbrace
\end{align*}
where 
\begin{equation*}
f(x,u)=\frac{e^{-\lambda(x)(x_b\wedge x_d)}}{G(x)}\frac{\overline{n}_x^1(0)}{2}\left(\int_{0}^{u}h^2 k(h)dh + \int_{u}^{1}u h k(h)dh\right).
\end{equation*}
\end{lemm}
\begin{proof}
Let $i\in\lbrace 1,2\rbrace$. Let $K$ be a compact subset of $U_i$. Let $\delta>0$. We fix $\epsilon_0>0$ such that for all $x\in K$, $\epsilon <\epsilon_0$ and $h\in\left[0,1\right]$, we have $x+\epsilon(h)_i\in U_i$. The map $\lambda$ is differentiable on $U_i$. Hence, for all $(x,\epsilon,h)\in K\times\left[0,\epsilon_0\right]\times\left[0,1\right]$, there exists $\theta\in \left[x,x+\epsilon (h)_i\right]$ such that 
\begin{equation*}
\lambda(x+\epsilon (h)_i)-\lambda(x) = \epsilon (h)_i.\nabla \lambda(\theta).
\end{equation*}
Let $x\in K$, we have
\begin{align*}
\left\Vert \frac{\tau g^{\epsilon}(x)}{\epsilon} - (f(x,1))_i\right\Vert & =\frac{\overline{n}^1_x(0)}{2} \left\vert \int_{\R_+}\left[\nabla \lambda(\theta)\right]_i h^2 k(h)dh - \left[\nabla \lambda(x)\right]_i\int_{\R_+}h^2 k(h)dh\right\vert \\
&\leq \frac{\sup_{x\in\mathcal{V}}\overline{n}_x^1(0)}{2}\int_{\R_+}\vert \left[\nabla \lambda(\theta)\right]_i - \left[\nabla \lambda(x)\right]_i\vert h^2 k(h)dh.
\end{align*}
By Proposition \ref{pr:proprietesmalthus}, the map $\nabla \lambda$ is Lipschitz on $U_i$ with some Lipschitz constant $C$. We deduce that
\begin{equation*}
\left\Vert \frac{\tau g^{\epsilon}(x)}{\epsilon} - (f(x,1))_i\right\Vert\leq \frac{\epsilon \sup_{x\in\mathcal{V}}\overline{n}_x^1(0)C}{2}\int_{\R_+}h^3k(h)dh.
\end{equation*}
We obtain that $\tau g^{\epsilon}/\epsilon$ converges uniformly on all compact subsets of $U_i$ and we conclude that for all $x\in U_i$, $H(x)=(f(x,1))_i$.
\\
Let $x\in \mathcal{H}$. We first show that 
\begin{equation}\label{eq:egsets}
\left\lbrace \text{acc}_{\epsilon\rightarrow 0}\frac{\tau g_{\epsilon}(x^{\epsilon})}{\epsilon}:x^{\epsilon}\rightarrow x\right\rbrace = \bigcup_{i\in\lbrace 1,2\rbrace}\left\lbrace (f(x,u))_i:u\in\left[0,1\right]\right\rbrace.
\end{equation}
If the fitness gradient was smooth (i.e Lipschitz), the sets in (\ref{eq:egsets}) would be reduced to one element and we would be in the cases studied in \cite{champagnat2002canonical}, \cite{champagnat2011polymorphic} or \cite{meleard2009trait}. We prove the inclusion from right to left. Let $i\in\lbrace 1, 2\rbrace$, let $u\in\left[0,1\right]$, we define the sequence $x^{\epsilon}= x -\epsilon(u)_i$. We have
\begin{align*}
\frac{\tau g^{\epsilon}(x^{\epsilon})}{\epsilon}& =\left(\int_{\R_+}\frac{\lambda(x -\epsilon(u)_i +\epsilon (h)_i) -\lambda(x-\epsilon(u)_i)}{\epsilon}k(h)(h)_i dh\right)\frac{\overline{n}_{x^{\epsilon}}^1(0)}{2}\\
&=\left(\int_{0}^{u}\frac{\lambda(x -\epsilon(u)_i +\epsilon (h)_i) -\lambda(x-\epsilon(u)_i)}{\epsilon}k(h)(h)_i dh \right.\\
& \left. + \int_{u}^{1}\frac{\lambda(x -\epsilon(u)_i +\epsilon (h)_i) -\lambda(x-\epsilon(u)_i)}{\epsilon}k(h)(h)_i dh\right)\frac{\overline{n}_{x^{\epsilon}}^1(0)}{2}.
\end{align*}
For all $h\in\left[0,u\right]$, we have $\left[x-\epsilon(u)_i,x - \epsilon(u-h)_i\right]\subset U_i$. So we can find $\theta \in \left[x-\epsilon(u)_i,x - \epsilon(u-h)_i\right]$ such that $\lambda(x -\epsilon(u-h)_i) -\lambda(x-\epsilon(u)_i) = \epsilon (h)_i.\nabla \lambda(\theta)$. By Proposition \ref{pr:proprietesmalthus} (ii), we deduce that 
\begin{equation*}
\int_{0}^{u}\frac{\lambda(x -\epsilon(u)_i +\epsilon (h)_i) -\lambda(x-\epsilon(u)_i)}{\epsilon}\text{ }k(h)(h)_i dh \rightarrow \frac{e^{-\lambda(x) (x_b\wedge x_d)}}{G(x)} \int_{0}^{u}(h^2)_i\text{ }k(h)dh
\end{equation*}
as $\epsilon$ tends to zero. For all $h\in\left[u,1\right]$ we have $\lambda(x +\epsilon(h-u)_i)=\lambda(x)$. We deduce similarly that 
\begin{equation*}
\int_{u}^{1}\frac{\lambda(x -\epsilon(u)_i +\epsilon (h)_i) -\lambda(x-\epsilon(u)_i)}{\epsilon}\text{ }k(h)(h)_i dh \rightarrow \frac{e^{-\lambda(x) (x_b\wedge x_d)}}{G(x)}\int_u^1 (h)_i u\text{ }k(h)dh
\end{equation*} 
as $\epsilon$ tends to zero. We conclude the proof of the first inclusion arguing that $\overline{n}_{x^\epsilon}^1(0)\rightarrow \overline{n}_x^1(0)$ as $\epsilon$ tends to zero. \\We prove the inclusion from left to right. Let $x\in \mathcal{V}$. If $x\in\mathcal{H}$ then we have 
\begin{equation}\label{eq:eg}
g^{\epsilon}(x)=0.
\end{equation} 
If $x\in U_i$, for some $i\in\lbrace 1,2\rbrace$, then we have
\begin{align*}
\frac{\tau g^{\epsilon}(x)}{\epsilon}=\frac{\overline{n}_{x}^1(0)}{2}\int_{0}^{1}\frac{\lambda(x+\epsilon (h)_i)-\lambda(x)}{\epsilon}k(h)(h)_i dh
\end{align*}
Moreover for all $\epsilon,h$ there exists $\theta\in\left[x,x+\epsilon(h)_1\right]$ such that $\lambda(x+\epsilon (h)_i)-\lambda(x)\leq \frac{\partial \lambda(\theta)}{\partial x_i}\epsilon h$. We deduce that 
\begin{equation}\label{eq:ineg}
\left[\frac{\tau g^{\epsilon}(x)}{\epsilon}\right]_i\leq \frac{\overline{n}_{x}^1(0)}{2}\int_{0}^{1}\frac{\partial \lambda(\theta)}{\partial x_i} h^2 k(h) dh
\end{equation}
From (\ref{eq:eg}) and (\ref{eq:ineg}), we deduce easily the second inclusion in (\ref{eq:egsets}). We conclude that
\begin{equation*}
\text{conv}\left\lbrace \text{acc}_{\epsilon\rightarrow 0}\frac{\tau g_{\epsilon}(x^{\epsilon})}{\epsilon}:x^{\epsilon}\rightarrow x\right\rbrace = \text{conv}\left\lbrace\cup_{i\in\lbrace 1,2\rbrace}\left\lbrace (f(x,u))_i:u\in\left[0,1\right]\right\rbrace\right\rbrace=H(x).
\end{equation*} 
\end{proof}
In Lemma \ref{le:eqmapset}, we prove that differential inclusions associated with $H$ and $F$ have identical solutions. Before, we give a technical lemma.
\begin{lemm}\label{lem:tech}
Let $x\in\mathcal{H}$. Let $(u,v,\alpha)\in\left[0,1\right]^3$ and let 
\begin{equation*}
m=\alpha \left(\begin{array}{c}
f(x,u) \\ 
0
\end{array}\right) +(1-\alpha)\left(\begin{array}{c}
0 \\ 
f(x,v)
\end{array}\right)\in H(x).
\end{equation*}
Assume that $m\notin F(x)$. Then we have $\alpha  f(x,u)\neq (1-\alpha) f(x,v)$.
\end{lemm}
\begin{proof}
We prove the lemma by contradiction. Assume that $\alpha  f(x,u)= (1-\alpha) f(x,v)$. Then we obtain 
\begin{equation*}
m=\frac{1}{2}\left(\begin{array}{c}
\frac{2f(x,u)f(x,v)}{f(x,u)+f(x,v)}\\ 
\frac{2f(x,u)f(x,v)}{f(x,u)+f(x,v)}\end{array}\right).
\end{equation*}
We assume without loss of generality that $f(x,u)\leq f(x,v)$. Then we obtain $0\leq\frac{2f(x,u)f(x,v)}{f(x,u)+f(x,v)}\leq f(x,v)$. Finally we remark that the map $s\mapsto f(x,s)$ is a bijection from $\left[0,1\right]$ to $\left[0,f(x,1)\right]$. Hence there exists $w\in\left[0,1\right]$ such that $f(x,w)=\frac{2f(x,u)f(x,v)}{f(x,u)+f(x,v)}$ that allows us to obtain the contradiction.
\end{proof}
\begin{lemm}\label{le:eqmapset}
Any solution of 
\begin{equation}\label{eq:difincF}
\begin{cases}
\frac{\text{d}x(t)}{\text{dt}}\in F(x(t)),\quad t\in\left[0,T\right]\\
x(0)=x^{0}.
\end{cases}
\end{equation}
is a solution of 
\begin{equation}\label{eq:difincH}
\begin{cases}
\frac{\text{d}x(t)}{\text{dt}}\in H(x(t)),\quad t\in\left[0,T\right]\\
x(0)=x^{0}.
\end{cases}
\end{equation}
and conversely.
\end{lemm}
\begin{proof}
For all $x\in\mathcal{V}$, we have $F(x)\subset H(x)$. We deduce that if $(x(t),t\in\left[0,T\right])$ is a solution of (\ref{eq:difincF}) then it is a solution of (\ref{eq:difincH}). Assume conversely that there exists a solution $(x(t),t\in\left[0,T\right])$ of (\ref{eq:difincH}) which is not a solution of (\ref{eq:difincF}). We deduce that there exists $t_0\in\left[0, T\right]$ such that $x$ is differentiable at $t_0$, $\text{d}x(t_0)/\text{dt}\in H(x(t_0))$ and $\text{d}x(t_0)/\text{dt}\notin F(x(t_0))$. Then we have  $x(t_0)\in\mathcal{H}$ and $\text{d}x(t_0)/\text{dt}\notin F(x(t_0))$.

We now deduce the contradiction. By Lemma \ref{lem:tech}, we obtain that
$\ \frac{\text{d}x_b(t_0)}{\text{d}t}\neq \frac{\text{d}x_d(t_0)}{\text{d}t}$. 
Without loss of generality, we may assume that 
$\, 
\frac{\text{d}x_b(t_0)}{\text{d}t} < \frac{\text{d}x_d(t_0)}{\text{d}t}$.
Since $x_b(t_0)=x_d(t_0)$, there exists an interval $\left]t_0,t_1\right[$ such that for all $s\in\left]t_0,t_1\right[$, $x(s)\in U_1$. Assume that for all $s\in\left]t_0,t_1\right[$ such that $x$ is differentiable at $s$, we have 
\begin{equation}\label{eq:deriv0}
\frac{\text{d}x_d(s)}{\text{d}t}=0.
\end{equation}
The solution $x$ of the differential inclusion (\ref{eq:difincH})  is absolutely continuous. Hence for all $s\in\left]t_0, t_1\right[$,  $x_d(s)=x_d(t_0)=x_b(t_0)$. Since  $x(s)\in U_1$, we obtain that $x_b(s)<x_b(t_0)$. It is absurd since $x_b$ is non-decreasing and hence it contradicts (\ref{eq:deriv0}). So, there exists $s_0\in\left]t_0,t_1\right[$ such that $x(s_0)\in U_1$, $x$ is differentiable at $s_0$  and satisfies
$\ 
\frac{\text{d}x_d(s_0)}{\text{d}t}>0$.
However, $x$ is a solution of (\ref{eq:difincH}) that leads to the final contradiction.
\begin{figure}[!h]
\begin{center}
\includegraphics[scale=0.2]{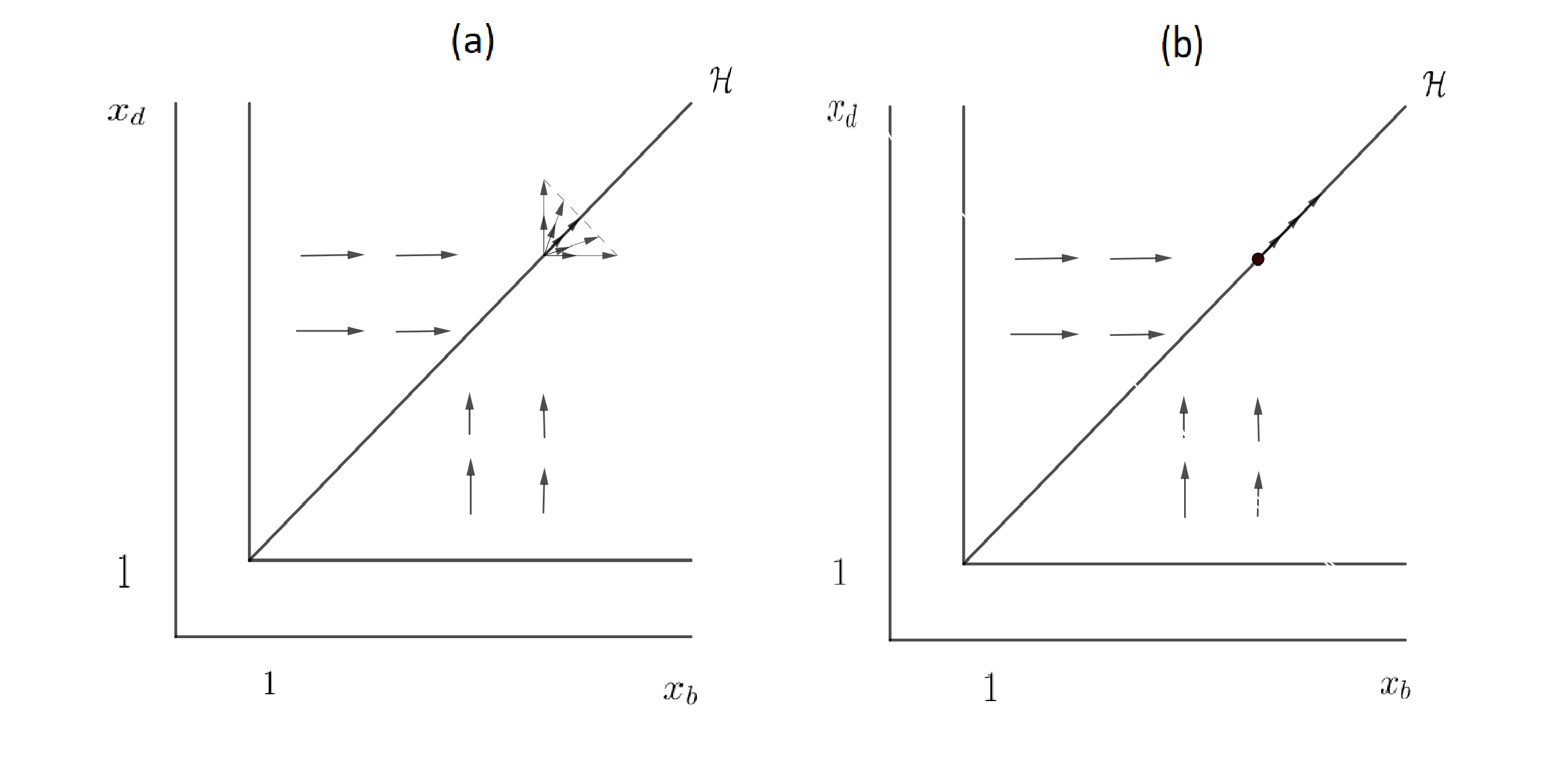}
\caption{(a): Representation of $H$. (b): Representation of $F$.}
\label{fig:dessinFH}
\end{center}
\end{figure}
\end{proof}
We now give the proof of Theorem \ref{theo:diffinc}. It is a direct consequence of \cite[Theorem 1]{gast2012markov} recalled in Appendix A.3.
\begin{proof}[of Theorem \ref{theo:diffinc}]
Let $T>0$ be fixed. By Lemma \ref{le:eqmapset} and Theorem \ref{th:gastgaujal}, we deduce that for all $\delta>0$ we have
\begin{equation}\label{eq:prtheo}
\lim_{\epsilon\rightarrow 0}\mathbb{P}\left( \inf_{x\in \mathcal{S}_F(T,x^{0})}\sup_{t\in\left[0,T\right]}\vert Y^{\epsilon}(\lfloor t\tau/\epsilon\rfloor) -x(t)\vert >\delta\right) = 0.
\end{equation}
We conclude by using similar arguments as in the proof of \cite[Theorem 4]{gast2012markov}. Since $\Lambda^{\epsilon}$ is a Poisson process with parameter $\tau/\epsilon$, we obtain that for all $\delta>0$,
\begin{equation}\label{th:pr3}
\mathbb{P}\left(\sup_{t\leq T}\vert \Lambda^{\epsilon}(t) -\frac{t\tau}{\epsilon}\vert \geq \frac{\tau\delta}{\epsilon}\right)\leq \frac{T\epsilon}{\tau \delta}.
\end{equation}
We have $\mathbb{P}\left(\inf_{x\in \mathcal{S}_F(T,x^0)}\sup_{t\leq T}\vert X^{\epsilon}(t) - x(t)\vert > \delta\right)$
\begin{align}\label{th:pro2}
&=\mathbb{P}\left(\inf_{x\in \mathcal{S}_F(T,x^0)}\sup_{t\leq T}\vert Y^{\epsilon}(\Lambda^{\epsilon}(t)) - x(t)\vert > \delta\right)\nonumber\\
&\leq \mathbb{P}\left(\inf_{x\in \mathcal{S}_F(T,x^0)}\sup_{t\leq T}\left\lbrace\vert Y^{\epsilon}(\Lambda^{\epsilon}(t))-x\left(\frac{\epsilon \Lambda^{\epsilon}(t)}{\tau}\right)\vert + \vert x\left(\frac{\epsilon \Lambda^{\epsilon}(t)}{\tau}\right) -x(t)\vert \right\rbrace>\delta\right).
\end{align}
Let $x\in\mathcal{S}_F(T,x^{0})$ be a solution of the differential inclusion (\ref{eq:difinc}). For almost all $t\in\left[0,T\right]$, we have $\text{d}x(t)/\text{d}t\in F(x(t))$. Since $\sup_{x\in\mathcal{V}}\sup\lbrace F(x)\rbrace < +\infty$, we deduce that there exists $C_T>0$ such that for all $y\in\mathcal{S}_F(T,x^{0})$, for all $s,t\in \left[0,T\right]$, $\vert y(t) - y(s)\vert \leq C_T \vert t - s\vert$. We deduce that (\ref{th:pro2}) is less than
\begin{align*}
&\mathbb{P}\left(\inf_{x\in \mathcal{S}_F(T,x^0)}\left\lbrace\sup_{t\leq T}\vert Y^{\epsilon}(\Lambda^{\epsilon}(t))-x\left(\frac{\epsilon \Lambda^{\epsilon}(t)}{\tau}\right)\vert\right\rbrace +  C_T \sup_{t\leq T}\vert \frac{\epsilon \Lambda^{\epsilon}(t)}{\tau} - t\vert >\delta\right)\\
&\leq \mathbb{P}\left(\inf_{x\in \mathcal{S}_F(T,x^0)}\sup_{t\leq T}\vert Y^{\epsilon}(\Lambda^{\epsilon}(t))-x\left(\frac{\epsilon \Lambda^{\epsilon}(t)}{\tau}\right)\vert>\delta\right) +  \mathbb{P}\left( C_T \sup_{t\leq T}\vert \frac{\epsilon \Lambda^{\epsilon}(t)}{\tau} - t\vert >\delta\right)
\end{align*}
and we conclude by using (\ref{eq:prtheo}) and (\ref{th:pr3}).
\end{proof}

\section{Discussion}
\subsection{General discussion and comments}
In the present article, we study the evolution of a population with a trait structure describing a simple class of life-histories. We build a stochastic individual-based model in a framework that is continuous in time, age and trait. The trait is a pair of parameters $(x_b,x_d)$ characterising the age at end-of-reproduction $x_b$ and the age at transition to a non-zero mortality risk $x_d$.  The model sees two origins of phenotypic variation. First, the genetic mutations that are supposed to be rare and modify the traits symmetrically - equal probability to increase or decrease the value of the parameter. Second, we model the Lansing effect which can be considered as an epigenetic mutation affecting the progeny of an “old” individual. It is acting on a much faster time-scale - one generation - than genetic mutations and has only a negative effect on the life expectancy of the progeny. We highlight that we have chosen to model the Lansing effect by an extremely strong effect on the descendant. Indeed, it acts on each generation and degrades dramatically the life-expectancy of the progeny. Nevertheless, a more general and realistic study of epigenetic modifications on the genetic evolution has to be investigate. Some aspects of this question have been studied in \cite{klironomos2013epigenetic}. The reasoning is based on the fact that epigenetic modifications are more frequent than genetic mutations (\cite{schmitz2011transgenerational}). It would be interesting to extend adaptive dynamics tools in order to take into account epigenetics modifications (involving intermediary time scales).
\\We study the long term evolution of the trait distribution using adaptive dynamics theory. We extend the age-structured TSS to take into account the Lansing effect. The main mathematical result of the present work concerns the behaviour of the TSS when mutations are small. We show in Theorem \ref{theo:diffinc} that the behaviour of the rescaled TSS in the limit $\epsilon\rightarrow 0$ is characterised by a differential inclusion whose solutions are not unique on the diagonal $\mathcal{H}=\lbrace x_b=x_d\rbrace$. This differential inclusion allows to generalize the Canonical Equation of Adaptive Dynamics \cite{dieckmann1996dynamical},\cite{champagnat2011polymorphic} in order to consider the non-smooth fitness gradient. The proof is based on \cite{gast2012markov}. Thanks to this approach, we show that the evolution of our model, whatever its initial configuration, leads to the apparition and maintenance of configurations $(x_b,x_d)$ satisfying $x_b-x_d=0$.
\\To our knowledge, differential inclusions have never been used before in the adaptive dynamics theory. In \cite{champagnat2002canonical}, \cite{champagnat2011polymorphic},  \cite{meleard2009trait}, the fitness gradient is assumed to be a Lipschitz function, which ensures the uniqueness of the solutions. Our approach seems useful for generalizing the canonical equation to situations where the fitness gradient is neither Lipschitz nor continuous.  
\\The drift associated with the differential inclusion depends  on the fitness gradient that satisfies (see Proposition \ref{pr:proprietesmalthus}) 
\begin{equation}\label{eq:fitnessgraddiscuss}
\forall i\in\lbrace 1,2\rbrace,\quad \forall x\in U_i,\quad \nabla \lambda(x)=\left(\frac{e^{-\lambda(x)(x_b\wedge x_d)}}{G(x)}\right)_i
\end{equation}
where $\lambda(x)$ is the Malthusian parameter and $G(x)$ the mean generation time associated to the trait $x$. Hence the fitness gradient $\nabla \lambda(x)$ describes the speed of evolution of the trait $x$. It can be related to the seminal work of Hamilton (\cite{hamilton1966moulding}) on the moulding of senescence. In that article, Hamilton states that senescence is unavoidable because the \textit{strength of selection} decreases with age. To show it, he defines the strength of selection at some given age $a_0$ as the sensitivity of the Malthusian parameter with respect to some little perturbation on the birth or death intensities at age $a_0$. He concludes arguing that these quantities decrease to zero as $a_0$ tends to infinity. Similarly, formula (\ref{eq:fitnessgraddiscuss}) describes the sensitivity of the Malthusian parameter with respect to some perturbation on the duration of the reproduction or the survival phase at ages $x_b$ and $x_d$. Then, they can be interpreted as the strength of selection at ages $x_b$ and $x_d$ and describe the speed of evolution of the traits $x_b$ and $x_d$ in the canonical inclusion (\ref{eq:difinc}).

\vspace{0.2cm}
To conclude, the present article studies a  case of  birth and death model with a strong Lansing effect and constant competition applied to asexual and haploid individuals in order to validate mathematically the convergence of $x_b$ and $x_d$ observed in the numerical simulations. Further characterisation of this model is in progress, in order to better understand the influence of its different parameters on the evolution of $(x_b, x_d)$.
\\Our initial motivation for developing the bd-model was an attempt to understand whether a phenomenon leading to a dramatic decrease of an individual’s fitness could be selected through evolution with simple and no explicitely constraining trade-offs. Indeed, in the past years, Rera and collaborators have identified and characterised a dramatic transition  preceding death in drosophila (\cite{rera2011modulation}, \cite{rera2012intestinal}) as well as other organisms (\cite{rera2018smurf}, \cite{dambroise2016two}). We show here that, under uniform competition - i.e. environmental limitation equally affecting all genotypes - a mechanism coupling the end of reproductive capabilities and organismal homeostasis can and will be selected thanks to evolution. Thus, at least these two characteristics of senescent organisms can positively be selected through evolution. Concerning the biological interpretation of our model, our thesis is that individuals with a senescence mechanism associated to the Lansing effect tend to produce more genetic variants than those without senescence. Hence, these individuals could show a higher evolvability. This question is being investigated in a work in progress.
\\Here we show the positive selection of a property that limits reproduction. It is reminiscent of the article by \cite{tully2011evolution} 
who proposed a new selective mechanism for post-reproductive life span. The latter relies on the hypothesis that the post-reproductive lifespan can be selected as an insurance against indeterminacy; a longer life expectancy reducing the risk of dying by chance before the cessation of reproductive activity. In the present article, the maintenance of individuals showing Lansing effect is the counterpart of individuals with post-reproductive survival. As discussed in \cite{kirkwood1982cytogerontology}, one of August Weismann’s concepts that  persisted without changes throughout his life is a conviction that “life is endowed with a fixed duration, not because it is contrary to its nature to be unlimited, but because unlimited existence of individuals would be a luxury without any corresponding advantage” (\cite{bhl23551}). This is what we showed in the present work.

\subsection{Generalizations of the model}
\subsubsection{Lower Lansing Effect}\label{sec:lansing}
In this article, we model the Lansing effect by a very strong effect since we assume that the descendants of old individuals with trait $(x_b,x_d)$ inherits of the trait $(x_b,0)$. It could be interesting to study a more general case like $(x_b,\alpha x_d)$ for some $\alpha\in\left[0,1\right]$. 
\subsubsection{Mutation kernel}

It would be interesting to consider two different mutation kernels $k_b(x,h)$ and $k_d(x,h)$ for the traits $x_b$ and $x_d$.  This change should not modify  the behaviour of the process. Indeed, on the sides off the diagonal $(x_b \neq x_d)$, only the speed of evolution of the traits will be modified. Hence, the trait $(x_b,x_d)$ will converge to the diagonal and will evolve on it at some speed depending on the variances of the kernels $k_b$ and $k_d$. This case is being studied in an ongoing work mentionned in Section \ref{sec:lansing}.

\appendix
\section{Appendix}
\subsection{Proof of Proposition \ref{le:linearlong}}
The proof is based on classical arguments of spectral theory for strongly
continuous semi-groups. Let us denote by $(P_t)_{t\geq 0}$ the semi-group on $L^1(\R_+)^2$ associated with the infinitesimal generator $(A,D(A))$. Let us denote by $\sigma(A)$ and $\sigma_e(A)$ the spectrum and the essential spectrum of the operator $A$ respectively. Let us denote by $\alpha\left[P_t\right]$ the measure of non-compactness of $P_t$ \cite[Definition 4.14 p 165]{webb1985theory}, and define $w_1(A):=\lim_{t\rightarrow\infty}t^{-1}\log(\alpha\left[P_t\right])$. We show that there exists $\omega>0$ such that
\begin{equation}\label{eq:spectralgap}
\max\left(w_1(A),\sup_{z\in(\sigma(A)\setminus \sigma_e(A))\setminus\lbrace \lambda(x)\rbrace}\mathcal{R}e(z)\right)<\omega<\lambda(x).
\end{equation}
By using arguments similar to \cite[Section 1]{pruss1981equilibrium}, we obtain that for all $t$ large enough
\begin{equation*}
\alpha\left[P_t\right]\leq e^{-t}.
\end{equation*}
and that
\begin{equation*}
w_1(A):=\lim_{t\rightarrow\infty}t^{-1}\log(\alpha\left[P_t\right])\leq -1.
\end{equation*}
By \cite[Proposition 4.13 p 170]{webb1985theory}, we obtain that $\sigma_e(A)\subset \lbrace z\in\mathbb{C}:\mathcal{R}e(z)\leq -1\rbrace$. Let $z \in \sigma(A)\setminus \sigma_e(A)$. Then there exists a non-zero $u\in L^1(\R_+,\mathbb{C})^2$ such that
\begin{equation}\label{eq:prospec}
\begin{cases}
-u'(a) -\textbf{D}_x(a)u = z u(a)\\
u(0)=\int_{\R_+}\textbf{B}_x(\alpha)u(\alpha)d\alpha.
\end{cases}
\end{equation}
By solving the first equation in (\ref{eq:prospec}) and by injecting the result in the second equation, we obtain that $u(0)$ satisfies
\begin{equation}\label{eq:proeigenvect}
u(0)= \textbf{F}(z)u(0)
\end{equation}
where 
\begin{equation}\label{eq:F}
\textbf{F}(z) = \int_{\R_+}\textbf{B}_x(a)\exp\left( -\int_{0}^{a}(\textbf{D}_x(\alpha)+ z \textbf{I})d\alpha\right)da.
\end{equation}
Equation (\ref{eq:proeigenvect}) admits a non-trivial solution $u(0)$ if and only if $\det \left[\textbf{F}(z)-\textbf{I}\right] =0$. Since the matrix $\textbf{F}(z)$ is triangular we have
\begin{equation}\label{eq:spectsolution}
\det \left[\textbf{F}(z)-\textbf{I}\right]= (\left[\textbf{F}(z)\right]_{11} - 1)(\left[\textbf{F}(z)\right]_{22}-1)
\end{equation}
where $\left[\textbf{F}(z)\right]_{11}=\int_{0}^{x_b\wedge x_d}e^{-za}da$ and $\left[\textbf{F}(z)\right]_{22}=\int_{0}^{x_b}e^{-(1+z)a}da$.
We deduce that the Malthusian parameter $\lambda(x)$ is the largest real solution of $\det\left[\textbf{F}(z)-\textbf{I}\right]=0$. By using similar analytical arguments as in the proof of \cite[Theorem 4.10]{webb1985theory}, we deduce that there exists only finitely many $z\in\mathbb{C}$ such that $\mathcal{R}e(z)>0$ and $\det \left[\textbf{F}(z)-\textbf{I}\right]=0$ which allows us to conclude for (\ref{eq:spectralgap}). We prove that $\lambda(x)$ is a simple eigenvalue of $A$ by showing that $\lambda(x)$ is a simple zero of the equation $\det \left[\textbf{F}(z)-\textbf{I}\right]=0$. Indeed we have 
\begin{equation*}
\frac{\text{d}\det \left[\textbf{F}(\lambda(x))-\textbf{I}\right]}{\text{d}\lambda}=\frac{\text{d}\left[\textbf{F}(\lambda(x))\right]_{11}}{\text{d}\lambda}(\left[\textbf{F}(\lambda(x))\right]_{22}-1)>0.
\end{equation*}
Let $N_x$ be a principal eigenvector associated with the eigenvalue $\lambda(x)$. By (\ref{eq:prospec}), we have
\begin{equation*}
N_x^1(a)=N_x^1(0)e^{-((a-x_d)\vee 0) -\lambda(x) a},\quad N_x^2(a)=N_x^2(0) e^{-(1+\lambda(x))a}
\end{equation*}
and Equation (\ref{eq:proeigenvect}) gives that 
\begin{equation*}
N_x^2(0)= \frac{\left[\textbf{F}(\lambda(x))\right]_{21}}{1-\left[\textbf{F}(\lambda(x))\right]_{22}}N_x^1(0).
\end{equation*}
We conclude for the convergence by using  arguments similar to proof of \cite[Theorem 4.9  p187]{webb1985theory}. 
\subsection{Proof of Lemma \ref{le:edo}}
\begin{proof}
Equation (\ref{eq:ode}) has the form $\text{d}u/\text{dt}=f(t,u)$ with $f(t,u)\rightarrow g(u)$ as $t\rightarrow\infty$. So (\ref{eq:ode}) is called an asymptotically autonomous differential equation (\cite{markus2016ii}, \cite{thieme1994asymptotically}) with the limit equation
\begin{equation}\label{eq:edoasymp}
\frac{\text{d}y(t)}{\text{dt}}= My(t) - \eta \Vert y(t)\Vert_1 y(t).
\end{equation}
We first show that any solution $y(t)$ of (\ref{eq:edoasymp}) started at $y(0)\in\R_+^{*}\times \R_+$ converges to a stationary state. In \cite{ackleh2007comparison}, the proof is given when $M$ is irreducible. We give a slightly different proof. Let $\overline{z}$ be defined in (\ref{eq:statioedo}). It is straightforward to prove that $\overline{z}$ is the eigenvector of $M$ associated with the simple eigenvalue $m_{11}$, which satisfy the condition $\Vert \overline{z}\Vert_1=m_{11}/\eta$. Since $z(0)\in\R_+^{*}\times\R_+$ there exists a positive constant $c(z(0))$ such that $e^{-m_{11}t}e^{Mt}z(0)\rightarrow c(z(0))\overline{z}$ as $t\rightarrow\infty$. We now write
\begin{equation*}
\frac{y_1(t)}{y_{2}(t)}=\frac{\left[e^{-m_{11}t}e^{Mt}y(0)\right]_1}{\left[e^{-m_{11}t}e^{Mt}y(0)\right]_2}\rightarrow\frac{\overline{z}_1}{ \overline{z}_2}.
\end{equation*}
We deduce that the $\omega$-limit set of any solution of (\ref{eq:edoasymp}) is a subset of $\Delta = \lbrace z\in\R_+^2:z_1=\frac{\overline{z}_1}{\overline{z}_2}z_2\rbrace$. We conclude by proving that any solution starting from $\Delta$ converges to $\overline{z}$. Let us consider such a solution (always denoted by $y(t)$). We have 
\begin{align*}
\frac{\text{d}y_1(t)}{\text{dt}}&=y_1(t)(m_{11} - \eta\Vert y(t)\Vert_1),\\
\frac{\text{d}y_2(t)}{\text{dt}}&=y_2(t)\left(m_{21}\frac{y_1(t)}{y_2(t)}+m_{22}- \eta\Vert y(t)\Vert_1\right).
\end{align*}
Since the $\omega$-limit set is an invariant subset, we deduce that 
\begin{equation*}
\frac{\text{d}y_2(t)}{\text{dt}}= y_2(t)(m_{11} - \eta\Vert y(t)\Vert_1),
\end{equation*}
that $\Vert y(t)\Vert_1\rightarrow m_{11}/\eta$  and $y(t)\rightarrow \overline{z}$ as $t\rightarrow \infty$. In order to conclude about the solutions of (\ref{eq:ode}) we use \cite[Theorem 1.2]{thieme1994asymptotically} arguing that $\overline{z}$ is an asymptotically stable equilibrium of (\ref{eq:edoasymp}) and that for any $y(0)\in\R_+^{*}$, the $\omega$-limit set of any solution $y(t)$ of (\ref{eq:ode}) started at $y(0)$ is not a subset of $\lbrace 0\rbrace\times\R_+$. The first claim is easily proved by showing that the Jacobian matrix has negative eigenvalues. For the second claim, let us assume it is not satisfied. Then $y_1(t)\rightarrow 0$ as $t\rightarrow \infty$. Let us show the contradiction. Let $\epsilon$ be sufficiently small and $t_0$ such that for any $t\geq t_0$
\begin{align*}
0<m_{11}-\epsilon &\leq m_{11}+\mathcal{D}_{11}(t)\leq m_{11} +\epsilon\\
m_{22}-\epsilon &\leq m_{22} +\mathcal{D}_{22}(t)\leq m_{22} +\epsilon <0\\
0<m_{12} -\epsilon &\leq m_{12}+\mathcal{D}_{12}(t)\leq m_{12}+\epsilon. 
\end{align*}
We introduce 
\begin{align*}
P^{\epsilon}(y_1,y_2)&=y_1(m_{11}-\epsilon -\eta(y_1+y_2))\\
Q^{\epsilon}(y_1,y_2)&=y_2(m_{22}+\epsilon -\eta(y_1+y_2))+(m_{12}+\epsilon)y_1
\end{align*}
and 
\begin{align*}
\mathcal{A}^{\epsilon}&=\left\lbrace y\in\R_+^2:y_1+y_2 \leq m_{11} -\epsilon\right\rbrace\\
\mathcal{B}^{\epsilon}&=\left\lbrace y\in\R_+^2:y_2\geq \frac{1}{2\eta}\left(m_{22}+\epsilon -\eta y_1 +\sqrt{(m_{22}+\epsilon -\eta y_1)^2 + 4\eta (m_{12} +\epsilon)y_1}\right)\right\rbrace.
\end{align*}
We deduce that there exists $t_1$ such that for any $t\geq t_1$, $\frac{\text{d}y_2(t)}{\text{dt}}<0$ on $\mathcal{B}^{\epsilon}$ and $y_1(t)<\epsilon$. We deduce that there exists $t_2$ such that for all $t\geq t_2$, $y_1(t)+y_2(t)\leq \frac{m_{11}-\epsilon}{\eta}$. So for all $t\geq t_2$, $\frac{\text{d}y_1(t)}{\text{dt}}\geq 0$ which is absurd. 
\end{proof}
\subsection{Differential inclusions}
In this appendix, we recall the results of \cite{gast2012markov} which concern the approximation of Markov chains by differential inclusions.
\\Let $\epsilon>0$ be a scale parameter. Let $(Y^\epsilon(k),k\in\N)$ be a Markov chain with values in $\R^d$. The drift of the Markov chain $Y^{\epsilon}$ is defined by 
\begin{equation*}
g^{\epsilon}(x)=\mathbb{E}\left[Y^{\epsilon}(k+1)-Y^{\epsilon}(k)\vert Y^{\epsilon}(k)=x\right].
\end{equation*}
Let $(\gamma^\epsilon)_{\epsilon>0}$ be such that $\lim_{\epsilon\rightarrow 0}\gamma^\epsilon = 0$ and let us denote
\begin{equation*}
f^{\epsilon}(x)=\frac{g^{\epsilon}(x)}{\gamma^\epsilon}.
\end{equation*}
One can write the evolution of the Markov chain as a stochastic approximation algorithm with constant step size $\gamma^\epsilon$
\begin{equation*}
Y^{\epsilon}(k+1)=Y^{\epsilon}(k)+\gamma^\epsilon\left(f^{\epsilon}(Y^{\epsilon}(k))+U^{\epsilon}(k+1)\right)
\end{equation*}
where $U^{\epsilon}$ is a martingale difference sequence with respect to the filtration associated with the process $Y^{\epsilon}$. 
\\Let us define 
\begin{equation*}
F(x)=\text{conv}\left(\left\lbrace \text{acc}_{\epsilon\rightarrow 0}f^{\epsilon}(x^{\epsilon})\text{ for all }x^{\epsilon}\text{ such that }\lim_{\epsilon\rightarrow 0}x^{\epsilon} =x\right\rbrace\right)    
\end{equation*}
where $\text{conv}(A)$ denotes the convex hull of the set $A$ and $\text{acc}_{\epsilon\rightarrow 0}f^{\epsilon}(x^{\epsilon})$ denotes the set of accumulation points of the sequence $f^{\epsilon}(x^{\epsilon})$ as $\epsilon\rightarrow 0$.
Let us denote by $\mathcal{S}_F(T,x^0)$ the set of solutions $(x(t), t\in\left[0,T\right])$ of the differential inclusion
\begin{equation}\label{eq:diffincannex}
\begin{cases}
\frac{\text{d}x(t)}{d\text{t}}\in F(x(t)),\quad t\in\left[0,T\right]\\
x(0)=x^{0}.
\end{cases}
\end{equation}
Let us recall the definition of  a solution of (\ref{eq:diffincannex}).
\begin{defi}
A map $x:\left[0,T\right]\mapsto \R^d$ is a solution of (\ref{eq:diffincannex}) if there exists a map $\varphi:\left[0,T\right]\mapsto \R^d$ such that:
\begin{itemize}
\item[(i)]For all $t\in\left[0,T\right]$, $\,x(t)=x^{0} +\int_{0}^{t}\varphi(s)ds$,
\item[(ii)]For almost every $t\in\left[0,T\right]$, $\varphi(t)\in F(x(t))$.
\end{itemize}
\end{defi}
In particular (i) is equivalent to saying that $x$ is absolutely continuous. (i) and (ii) imply that $x$ is differentiable at almost every $t\in \left[0,T\right]$ with $\text{d}x(t)/\text{d}t\in F(x(t))$.

\vspace{0.2cm}
We define the continuous process $\overline{Y}^{\epsilon}(t)$ as the piecewise interpolation of $Y^{\epsilon}$ whose time has been accelerated by $1/\gamma^{\epsilon}$: for all $k\in\mathbb{N}$, $\overline{Y}^{\epsilon}(k\gamma^{\epsilon})=Y^{\epsilon}(k)$ and $\overline{Y}^{\epsilon}$ is linear on $\left[k\gamma^{\epsilon}, (k+1)\gamma^{\epsilon}\right]$. We have the following theorem proved in \cite[Theorem 1]{gast2012markov}.
\begin{theo}\label{th:gastgaujal}
Assume that:
\begin{itemize}
\item[•]There exists a constant $c>0$ such that for all $y\in \R^d,\quad \Vert f^{\epsilon}(y)\Vert \leq c(1 +\Vert y\Vert)$.
\item[•]$U^{\epsilon}$ is a martingale difference sequence which is uniformly integrable.
\end{itemize}
If $Y^{\epsilon}(0)$ tends to $x^{0}$ in probability as $\epsilon$ tends to zero, then
\begin{equation*}
\inf_{y\in \mathcal{S}_F(T,x^{0})}\sup_{t\in\left[0,T\right]}\Vert \overline{Y}^{\epsilon}(t) - x(t)\Vert \longrightarrow 0
\end{equation*}
in probability as $\epsilon$ tends to zero.
\end{theo}

\subsection{Numerical simulation}
We give the Python script for the numerical simulations of the individual based model described in Section 2. The algorithm is based on a classical acceptation/reject method (\cite{champagnat2006unifying}, \cite{tran2006modeles})

\vspace{0.2cm}
\scriptsize{
\begin{lstlisting}

#parameters

#intensity of competition
competition = 0.0005  

#probability of mutation
p_mut = 0.05                    

#variance of mutations
var_mut = 0.05					

#number of jumps
number_of_jumps = 1000000           

#initial population size
population_size = 10000		


#traits in the initial population

trait = numpy.ones((population_size, 2))
for k in range(len(trait[:, 0])):
    trait[k, :] = [1.2, 1.6]          


#initial matrix population: [x_b,x_d, living/dead, birth date, death date, id, id parent, parent senescent]


population_0 = numpy.zeros((len(trait[:, 0]), 8), order='C', dtype=numpy.float32)
population_0[:, 0:2] = trait         
population_0[:, 2] = 1                      
population_0[:, 3: 5] = 0  
for l in range(len(population_0[:, 0])):
    population_0[l, 5] = l+1             


#birth rates

def b(x, a):
    if a <= x[0]:
        u = 1       
    else:
        u = 0.0
    return u

#death rates

def d(x, a):
    if a > x[1]:
        u = 1         
    else:
        u = 0.0
    return u


# maximal jump intensity by individual  (for the acceptation/reject method)       

intmax = 2.0 + competition 



#Lansing effect

def lansing_effect(x, u):
    if x[1] > 0:                                                               
        if u < x[1]:                                                           
            r = x                                                              
        else:                                                                  
            r = [x[0], 0]
    else:                                                                      
        r = x
    return r


#mutation kernel

def dm(x):
    g = random.gauss(0, var_mut)
    while x + g < 0:
        g = random.gauss(0, var_mut)
    return x + g


#functions for acceptation/reject method

def acceptation_rejet_clone(t, a, n):
    return b(t, a)*(1-(p_mut*p_mut+2*p_mut*(1-p_mut)))/(intmax*n)

def acceptation_rejet_mutant_1(t, a, n):
    return acceptation_rejet_clone(t, a, n)+(b(t, a)*2*p_mut*(1-p_mut))/(intmax*n)

def acceptation_rejet_mutant_2(t, a, n):
    return acceptation_rejet_mutant_1(t, a, n) + (b(t, a)*p_mut*p_mut)/(intmax*n)

def acceptation_rejet_mort(t, a, n):
    return acceptation_rejet_mutant_2(t, a, n) + (d(t, a) + (n-1)*intensite_competition)/(intmax*n)



#transition of the process: it is based on a classical acceptation/reject method

def transition(p, time):

#initialisation for acceptation/reject method
	
#living individuals	
	viv = p[(p[:, 2] == 1), :]     
    
#population size    
    n = len(viv[:, 0])                                                   
    
#jump time    
    jump_time = random.expovariate(1)/(intmax*n*n)        
    
#uniform law on (0,1)    
    u = random.uniform(0, 1)    
    
#random sampling of one individual   
    ind = random.randint(0, n-1)   
    
    
    w = time - viv[ind, 3] + jump_time #age of this individual
	
#acceptation/reject method        
    while u > acceptation_rejet_mort(viv[ind, :2], w, n):          
        jump_time += (random.expovariate(1)/(intmax*n*n))                
        u = random.uniform(0, 1)                                             
        ind = random.randint(0, n-1)                                         
        w = time - viv[ind, 3] + jump_time
	
#when accepted
    s = viv[ind, :]                                                          
    a1 = acceptation_rejet_clone(s[:2], w, n)                                
    a2 = acceptation_rejet_mutant_1(s[:2], w, n)                             
    a3 = acceptation_rejet_mutant_2(s[:2], w, n)                             

    a = 0                                                                    

    if w > s[1]:                                                             
        a = 1
	
#Lansing effect
    s[:2] = lansing_effect(s[:2], w)

#if clonal birth
    if u <= a1:                                                                                                                                             
        c1 = [1, time + jump_time, 0.0, p[-1, 5]+1, s[5], a]                                                                                      
        p = numpy.vstack((p, numpy.append(numpy.array(s[:2]), numpy.array(c1))))                                                      
	
#if birth with mutation (1 trait)    
    elif a1 < u <= a2:                                                                                                                                      
        x1 = int(numpy.random.randint(1, 3, 1))                                                                                                                  
        z1 = [[dm(s[0]), s[1]], [s[0], dm(s[1])]]                                                       
        c2 = [1, time + jump_time, 0.0, p[-1, 5]+1, s[5], a]                    
        p = numpy.vstack((p, numpy.append(numpy.array(z1[x1-1]), numpy.array(c2))))                                                        
    
#if birth with mutation (2 traits)     
    elif a2 < u <= a3:                                                                                               
        c3 = [1, time + jump_time, 0.0, p[-1, 5]+1, s[5], a]
        p = numpy.vstack((p, numpy.append(numpy.array([dm(s[0]), dm(s[1])]), numpy.array(c3))))
    
#if death    
    else:                                                                                              
        p[(p[:, 5] == s[5]), 2] = 0                                                                      
        p[(p[:, 5] == s[5]), 4] = time + jump_time            

    time += jump_time

    return p, time


#for simulating one trajectory of the process of number_of_jumps jumps

def trajectoire(p):
    population_size = sum(p[:, 2])						time = 0                                       
    for cs in range(0, number_of_jumps):
        if population_size >= 1:
            x = transition(p, time)                
            p = x[0]                                
            time = x[1]                            
            population_size = sum(p[:, 2])
            print cs
        else:
            break
    p[(p[:, 4] == 0), 4] = time                   
    return p, time                               



\end{lstlisting}}

\paragraph{acknowledgements.} We acknowledge partial support by the Chaire Modélisation Mathématique et Biodiversité of Veolia Environment - \'Ecole Polytechnique - Museum National d'Histoire Naturelle - FX. We acknowledge partial support by CNRS and the ATIP/Avenir-Aviesan kick-starting group leaders program. This work was also supported by a public grant as part of the investissement d'avenir project, reference ANR-11-LABX-0056-LMH, LabEx LMH. Finally, we acknowledge the reviewers for their interesting and constructive suggestions.

\bibliographystyle{plain} 
\bibliography{bibli} 

\end{document}